\documentclass[twoside]{article}

\usepackage[margin=2cm]{geometry}
\usepackage{amsmath,amsthm,amssymb,latexsym}
\usepackage[v2,tips,curve]{xy}

\theoremstyle{plain}
\newtheorem{proposition}{Proposition}[section]
\newtheorem{theorem}[proposition]{Theorem}
\newtheorem{corollary}[proposition]{Corollary}
\newtheorem{lemma}[proposition]{Lemma}
\theoremstyle{definition}
\newtheorem{definition}[proposition]{Definition}
\newtheorem{remark}[proposition]{Remark}
\newtheorem{example}[proposition]{Example}
\newtheorem{conjecture}[proposition]{Conjecture}

\newcommand{\vnten}{{\overline\otimes}}
\newcommand{\mc}{\mathcal}
\newcommand{\mf}{\mathfrak}

\newcommand{\G}{\mathbb{G}}
\renewcommand{\H}{\mathbb{H}}
\newcommand{\K}{\mathbb{K}}

\newcommand{\ip}[2]{\langle{#1},{#2}\rangle}
\newcommand{\op}{{\operatorname{op}}}
\newcommand{\sap}{{\operatorname{SAP}}}
\newcommand{\kac}{{\operatorname{KAC}}}
\newcommand{\lin}{\operatorname{lin}}
\newcommand{\proten}{\widehat\otimes}
\newcommand{\CAP}{\operatorname{CAP}}

\begin{document}

\large
\title{Remarks on the Quantum Bohr Compactification}
\author{Matthew Daws}
\maketitle

\begin{abstract}
The category of locally compact quantum groups can be described as either
Hopf $*$-homomorphisms between universal quantum groups, or as bicharacters
on reduced quantum groups.  We show how So{\l}tan's quantum Bohr compactification
can be used to construct a ``compactification'' in this category.
Depending on the
viewpoint, different C$^*$-algebraic compact quantum groups are produced, but
the underlying Hopf $*$-algebras are always, canonically, the same.  We show
that a complicated range of behaviours, with C$^*$-completions between the
reduced and universal level, can occur even in the cocommutative case, thus
answering a question of So{\l}tan.  We also study such compactifications from
the perspective of (almost) periodic functions.  We give a definition of a
periodic element in $L^\infty(\G)$, involving the antipode, which allows one
to compute the Hopf $*$-algebra of the compactification of $\G$; we later study
when the antipode assumption can be dropped.  In the
cocommutative case we make a detailed study of Runde's notion of a completely
almost periodic functional-- with a slight strengthening, we show that for
[SIN] groups this does recover the Bohr compactification of $\hat G$.

2010 \emph{Mathematics Subject Classification:}
   Primary 43A60, 46L89; Secondary 22D25, 43A20, 43A30, 43A95, 47L25.

\emph{Keywords:} Bohr compactification, Completely almost periodic functional,
Locally compact quantum group.

\end{abstract}

\section{Introduction}

The Bohr, or strongly almost periodic, compactification of a topological group
$G$ is the maximal compact group $G^\sap$ containing a dense homomorphic image
of $G$.  One can construct $G^\sap$ by looking at the finite-dimensional
unitary representation theory of $G$, but when $G$ is locally compact, there
is an intriguing link with Banach and C$^*$-algebra theory.  Let $AP(G)$
denote the collection of $f\in C^b(G)$ whose orbits, under the left- or
right-translation actions of $G$ on $C^b(G)$, form relatively compact subsets
of $C^b(G)$ (the collection of almost periodic functions).
Then $AP(G)$ is a commutative unital C$^*$-algebra, the character
space is $G^\sap$, and the group structure of $G^\sap$ can be ``lifted'' from
the group structure of $G$.  In this picture, we never go near representation
theory!

In the framework of noncommutative topology, one replaces spaces by algebras--
we think of $G$ as being represented by $C_0(G)$, and the product on $G$ being
given by a coproduct $\Delta:C_0(G)\rightarrow M(C_0(G)\otimes C_0(G))
= C^b(G\times G); \Delta(f)(s,t) = f(st)$.  Then $\Delta$ is coassociative,
and one can think of a ``quantum semigroup'' (or, more prosaically, a
C$^*$-bialgebra) as being a pair $(A,\Delta)$ where $A$ need no longer be
commutative.  When $A$ is unital, and we have the ``cancellation'' conditions
that
\[ \lin\{ \Delta(a)(b\otimes 1) : a,b\in A\}, \quad
\lin\{ \Delta(a)(1\otimes b) : a,b\in A\} \]
are dense in $A\otimes A$, then we have a compact quantum group.
The pioneering work of Woronowicz, \cite{woro2}, shows that such objects have
a remarkable amount of structure, and generalise completely the theory of
compact groups.  So{\l}tan in \cite{soltan} studied how to ``compactify'' a
C$^*$-bialgebra, and produced a very satisfactory theory, very much paralleling
(and generalising) the representation--theoretic approach to constructing
the classical Bohr compactification.

Going back to a locally compact group $G$, more abstractly, we can work with
the convolution algebra $L^1(G)$, turn the dual space $L^1(G)^* = L^\infty(G)$
into an $L^1(G)$-bimodule, and look at the functionals
$F\in L^\infty(G)$ such that the orbit map
$L^1(G)\rightarrow L^\infty(G); a\mapsto a\cdot F$ is a compact linear map.
Then we also recover $AP(G) \subseteq C^b(G) \subseteq L^\infty(G)$.
This theory has been generalised to general Banach algebras, and in particular
to the Fourier algebra (firstly in \cite{dr1}).  However, links here with any
notion of a ``compactification'' are very tentative.

In this paper, we have two major goals, both centred around understanding
further So{\l}tan's construction as applied to locally compact quantum groups.
These are C$^*$-bialgebras with additional, ``group-like'', structure.
Firstly, in a category theoretic sense, we have the inclusion functor from the
category of compact groups to the category of (say) locally compact groups.
The Bohr compactification is the universal arrow to this functor
(see Section~\ref{sec:cat} below).  Building on work of Kustermans and Ng,
the recent paper \cite{mrw} gave a very satisfactory picture for what
morphisms between locally compact quantum groups should be.  In 
Section~\ref{sec:comp_lcqg} we show how to construct a compactification
as a ``universal object'' in this category, see
Proposition~\ref{prop:compact_lcqg}.  A major technical stumbling
block is that we think of a single locally compact quantum group as being
represented by a number of different algebras, this paralleling the fact
that for a non-amenable $G$, the universal group C$^*$-algebra $C^*(G)$, and
the reduced algebra $C^*_r(G)$, are different.  Working at the ``universal''
level, the morphisms for locally compact quantum groups are just Hopf
$*$-homomorphisms, but So{\l}tan's construction may fail to give a universal
compact quantum group.  Similarly, compactifying at the reduced level 
may give a different compact quantum group, but we show that the underlying
Hopf $*$-algebras are always the same (in a canonical way), see
Proposition~\ref{prop:same_red_uni}.

In Section~\ref{sec:ba_sec} we study our other major goal, and look at how
the quantum Bohr compactification could be constructed without reference to
(co)representations (thus paralleling the ``almost periodic'' construction
of the classical Bohr compactification).  For a locally compact quantum
group $\G$, the philosophy is that the ``group structure'' of $\G$ should be
enough to allow us to construct the compactification $\G^\sap$ without 
explicitly looking at representations.  We define a notion of a ``periodic''
element, and show how to recover this from just knowledge of the convolution
algebra $L^1(\G)$, see Proposition~\ref{prop:one}.  We then show that for
Kac algebras, or under a further hypothesis involving the antipode, this
notion of periodic element allows one to construct $\G^\sap$, see
Section~\ref{sec:pap}.

In Section~\ref{sec:cocomcase} we study the Fourier algebra is further detail.
In \cite{runde} Runde used Operator Space theory to define the notion of
a ``completely almost periodic functional''.  Under an injectivity hypotheses,
we end up looking at the C$^*$-algebra
\[ \{ x\in L^\infty(\G) : \Delta(x) \in L^\infty(\G)
\otimes L^\infty(\G) \}, \]
where $\otimes$ here denotes the C$^*$-algebraic spacial tensor product.
In the fully quantum case, we show in Section~\ref{sec:e2} that the quantum
$E(2)$ group gives an example to show that there is little hope of such a
definition capturing the Bohr compactification.  However, in
Theorem~\ref{thm:cap_cmpt_kac_case} we show, in particular, that for a
discrete group $G$ this definition, when applied to the Fourier algebra
$A(G)$, does recover $C^*_r(G)$ as we might hope; for the classical almost
periodic definition, this was only known in the amenable case.
We then study [SIN] groups, and show that a slight further strengthening of
Runde's definition does allow us to recover the quantum Bohr compactification,
see Theorem~\ref{thm:cap_partial_case}.

Finally, in Section~\ref{sec:egs}, we study further examples.  By looking
again at the Fourier algebra, we answer (negatively) some conjectures of
So{\l}tan, showing in particular that finding the quantum Bohr compactification
of $C^*(G)$ and $C^*_r(G)$ may yield different completions of the same
underlying Hopf $*$-algebra, and that even for the reduced $C^*_r(G)$, the
resulting compact quantum group might fail to be itself reduced.  This also
shows that we did indeed need to be careful in Section~\ref{sec:comp_lcqg}!
In the special cases of discrete and compact quantum groups, we show how
the ``extra hypotheses'' which appeared in previous sections can be removed.

We start the paper in Section~\ref{sec:cat} with an introduction to the
quantum groups we are interested in, and the categories they form.  We
finish the paper with some open problems.

\subsection{Acknowledgements}

We thank Yemon Choi for helpful comments on Section~\ref{sec:cat},
Biswarup Das for useful commments on an early version of this article,
Piotr So{\l}tan for bringing \cite{woro4} to our attention, and the anonymous
referee for careful proof-reading.
The author was partly supported by EPSRC grant EP/I026819/1.

\section{Categories}\label{sec:cat}

We take a slightly general approach to compactifications.
Let $\mf B$ be a category, and let $\mf C$ be a full subcategory of $\mf B$.
We shall think of the objects of $\mf C$ as being the ``compact'' objects
of $\mf B$ (but be aware that this has nothing to do with the, somewhat more
specific, category--theoretic notion of a ``compact object'').  Given an
object $B\in\mf B$, a ``compactification'' of $B$ is an object $C\in\mf C$
and an arrow $B\overset{f}{\rightarrow}C$ which satisfies the following
universal property: for any $C'\in\mf C$ and any arrow $B\rightarrow C'$, then
there is a \emph{unique} arrow $C\rightarrow C'$ making the diagram commute:
\[ \xymatrix{ B \ar[r] \ar[rd] & C' \\ & C \ar[u]_{!} } \]
In particular, taking $C=C'$, uniqueness ensures that the identity morphism
on $C$ is the only arrow $g:C\rightarrow C$ with $gf=f$.
This property ensures that compactifications, if they exist, are unique
up to isomorphism.  Indeed, suppose that $B$ has two compactifications,
$B\overset{f_0}{\rightarrow} C_0$ and $B\overset{f_1}{\rightarrow} C_1$.
Then applying the universal property
of $C_0$ to the arrow $f_1$ yields a unique $g_0:C_0\rightarrow C_1$ with
$g_0 f_0 = f_1$.  Similarly we get a unique $g_1:C_1\rightarrow C_0$ with
$g_1 f_1 = f_0$:
\[ \xymatrix{ & C_0 \ar[d]^{g_0} \\
B \ar[ru]^{f_0} \ar[r]^{f_1} \ar[rd]_{f_0} & C_1 \ar[d]^{g_1} \\
& C_0 } \]
Then the composition $g_1g_0$ satisfies the relation $g_1g_0f_0 = f_0$, and
so $g_1g_0$ is the identity.  Similarly, $g_0g_1$ is the identity, and so
$C_0$ and $C_1$ are isomorphic, as claimed.

Suppose now that every object in $\mf B$ has a compactification; so by
uniqueness, we get a map $\mc F:\mf B\rightarrow\mf C$.  Given any arrow
$B_0 \overset{f}{\rightarrow} B_1$ in $\mf B$, we have the composition
$B_0 \overset{f}{\rightarrow} B_1 \rightarrow \mc FB_1$, where $\mc FB_1\in
\mf C$, and so the the universal property of $\mc FB_0$ gives
a unique arrow $\mc F B_0 \overset{\mc F f}{\rightarrow} \mc F B_1$ making
the following diagram commute
\[ \xymatrix{ B_0 \ar[r]^-{f} \ar[d] & B_1 \ar[d] \\
\mc F B_0 \ar[r]^-{\mc F f} & \mc F B_1 } \]
It is a simple exercise in drawing diagrams, and using uniqueness again, that
$\mc F(f\circ g) = \mc Ff\circ \mc Fg$, that is,
$\mc F$ is a functor $\mf B\rightarrow\mf C$.

Of course all this is well-known: our notion of
a compactification is just the ``universal arrow from $B$ to the inclusion
functor $\mf C\rightarrow\mf B$'' (see \cite[Chapter~III]{ml}).  Indeed,
if compactifications exist, then we have that $\mf C$ is ``reflective''
in $\mf B$ and the compactification is simply the ``reflection''
(see \cite[Section~IV.3]{ml}).  This sort of ``categorical'' approach to
defining the classical Bohr compactification of a group was stuided in
\cite{holm,holm1}, and for a similar treatment of the quantum case,
see the recent paper \cite{chir} (which essentially gives a treatment of
So{\l}tan's work via abstract categorical arguments, but which does not
consider the category {\sf LCQG} described below).

We next introduce the two categories which shall interest us in this paper.

\subsection{C$^*$-bialgebras}

Recall that a ``morphism'' (in the sense of Woronowicz) between C$^*$-algebras
$B$ and $A$ is a non-degenerate $*$-homomorphism $\theta:B\rightarrow M(A)$.
Such a non-degenerate $*$-homomorphism has a unique
extension to a unital, strictly continuous $*$-homomorphism $M(B)
\rightarrow M(A)$, the \emph{strict extension} of $\theta$,
which in this paper we shall tend to denote by the same
symbol $\theta$.  As such, morphisms can be composed.  We also tend to be
slightly imprecise, and to speak of a morphism from $B$ to $A$ (when really
the map is to $M(A)$) especially when drawing commutative diagrams.

The motivation comes from the commutative situation: if $A$ and $B$ are
commutative, then there are locally compact Hausdorff spaces $X_A,X_B$
with $A\cong C_0(X_A)$ and $B\cong C_0(X_B)$.  Furthermore, there is a
bijection between morphisms $\theta:B\rightarrow A$ and continuous maps
$\phi:X_A\rightarrow X_B$ given by $\theta(f) = f\circ\phi$.  If we did not
consider the multiplier algebra $M(A)$, then we would have to restrict
attention to proper continuous maps.  See \cite{lance,woro3} and perhaps
especially \cite[Chapter~2 exercises]{wo} for further details.

Let ${\sf CSBa}$ be the category of C$^*$-bialgebras $(A,\Delta)$,
here thought of in the general sense as $A$ being a (not necessarily
unital) C$^*$-algebra and $\Delta$ a non-degenerate $*$-homomorphism
$A\rightarrow M(A\otimes A)$ which is coassociative in the sense that
$(\Delta\otimes\iota)\Delta = (\iota\otimes\Delta)\Delta$.  An arrow
$(A,\Delta_A)\rightarrow (B,\Delta_B)$ is then a non-degenerate
$*$-homomorphism $\theta:B\rightarrow M(A)$ with $\Delta_A \theta =
(\theta\otimes\theta)\Delta_B$.  A non-degenerate $*$-homomorphism which
intertwines the coproducts in this fashion is
termed a ``Hopf $*$-homomorphism'' in \cite{mrw}.
Here we have ``reversed'' the arrows to generalise better from the commutative
situation, as if $(A,\Delta)\in{\sf CSBa}$ with $A$ commutative, then
$A=C_0(S)$ for some locally compact semigroup $S$ with $\Delta$ induced in
the usual way, $\Delta(f)(s,t) = f(st)$.  Given the discussion above, morphisms
restrict to the usual notion of a continuous semigroup homomorphism.

In ${\sf CSBa}$, we shall define the ``compact'' objects to be the compact
quantum groups in the Woronowicz sense, see the introduction and of course
\cite{woro2}.

\subsection{Locally compact quantum groups}

Let ${\sf LCQG}$ be the category of locally compact quantum groups,
with morphisms in the sense of \cite{mrw}.  Let us remark briefly on
definitions.  We shall follow the Kustermans and Vaes definition, see
\cite{kusbook, kv, kvvn, vaesphd}.

A \emph{locally compact quantum group} in the von Neumann algebraic setting
is a Hopf-von Neumann algebra $(M,\Delta)$ equipped with left and right
invariant weights.  As usual, we use $\Delta$ to turn $M_*$ into a Banach
algebra, and we write the product by juxtaposition.
We shall ``work on the left''; so using the left
invariant weight, we build the GNS space $H$, and a ``multiplicative''
unitary $W$ acting on $H\otimes H$ (of course, the existence
of a right weight is needed to show that $W$ is unitary).  There is a
(in general unbounded) antipode $S$ which admits a ``polar decomposition''
$S=R \tau_{-i/2}$, where $R$ is the unitary antipode, and $(\tau_t)$ is
the scaling group.  There is a nonsingular positive operator $P$ which
implements $(\tau_t)$ as $\tau_t(x) = P^{it} x P^{-it}$.  Then $W$ is
\emph{manageable} with respect to $P$.
One can develop a slightly more general theory of quantum group from
such manageable (or related, ``modular'') multiplicative unitaries,
see \cite{sw2, sw, woro}.  Many of the results of this paper work in this
more general setting; see remarks later.

Given such a $W$, the space $\{ (\iota\otimes\omega)W : \omega\in\mc B(H)_* \}$
is an algebra, and its closure is a C$^*$-algebra, say $A$.  There is a
coassociative map $\Delta:A\rightarrow M(A\otimes A)$ given by $\Delta(a)
= W^*(1\otimes a)W$.  If we formed $W$ from $(M,\Delta)$ with invariant weights,
then $A$ is $\sigma$-weakly dense in $M$, and the two definitions of $\Delta$
agree.  Similarly, $\{ (\omega\otimes\iota)W : \omega\in\mc B(H)_* \}$ is norm
dense in a C$^*$-algebra $\hat A$, and defining $\hat\Delta(\hat a) =
\hat W^*(1\otimes \hat a)\hat W$, we get a non-degenerate
$*$-homomorphism $\hat\Delta:\hat A \rightarrow M(\hat A\otimes \hat A)$,
where here $\hat W = \Sigma W^*\Sigma$, and $\Sigma$ is the tensor swap map
$H\otimes H \rightarrow H\otimes H$.  If we started with $(M,\Delta)$ having
invariant weights, then we can construct invariant weights on $(\hat A'',
\hat\Delta)$.  The unitary $W$ is in the multiplier algebra $M(A\otimes \hat A)
\subseteq \mc B(H\otimes H)$.

We write $\G$ for an abstract object to be thought of as a quantum group.
We write $C_0(\G), L^\infty(\G)$ and $L^1(\G)$ for $A,M$ and $M_*$.
We also write $L^2(\G)$ for $H$.

Again, the commutative examples always arise from locally compact groups with
their Haar measures, and again, morphisms between C$^*$-alegbras, which
intertwine coproducts, correspond to continuous group homomorphisms.  However,
the cocommutative examples are of the form $C^*_r(G)$, the \emph{reduced}
group C$^*$-algebras (and at the von Neumann algebra level, $VN(G)$) as we
work with faithful weights.  Then, for example, the trivial group homomorphism
should correspond to trivial representation $C^*(G)\rightarrow\mathbb C$,
but this remains bounded on $C^*_r(G)$ only for amenable $G$ (see
\cite[Theorem~4.21]{pat}, \cite[Section~6]{hul}).

The passage from $C^*_r(G)$ to $C^*(G)$ can analogously be performed for
quantum groups, see \cite{kus1}.  We shall write $C_0^u(\G)$ for the
\emph{universal} C$^*$-algebraic form of $\G$.
A similar object can be found for manageable multiplicative unitaries, see the
second part of \cite{sw}.

One possible definition for a morphism in ${\sf LCQG}$ is then a non-degenerate
$*$-homomorphism $C_0^u(\G)\rightarrow C_0^u(\H)$ which intertwines the
coproduct.  This was explored in \cite[Section~12]{kus1} where links with
certain coactions of the associated $L^\infty$ algebras was established.
Previously (before the canonical definition of a locally compact quantum
group had been given) Ng studied similar ideas in \cite{ng}.  Unifying and
extending these ideas, the paper \cite{mrw} shows that
Hopf $*$-homomorphisms at the universal level correspond bijectively to various
other natural notions of ``morphism'', and we take this as a working definition
of an arrow in ${\sf LCQG}$.

In summary, an object $\G$ in ${\sf LCQG}$ is
a locally compact quantum group, thought of as either being in reduced
form $C_0(\G)$, or in universal form $C_0^u(\G)$.  A morphism $\G\rightarrow
\H$ can be described in a number of equivalent ways:
\begin{itemize}
\item As a non-degenerate $*$-homomorphism $C_0^u(\H)\rightarrow
C_0^u(\G)$ which intertwines the coproduct (that is, a Hopf $*$-homomorphism
between universal quantum groups).
\item As a \emph{bicharacter} which is a unitary $U\in M(C_0(\G) \otimes
C_0(\hat\H))$ which satisfies $(\Delta_{\G} \otimes \iota)(U) = U_{13} U_{23}$
and $(\iota\otimes\Delta_{\hat\H})(U) = U_{13} U_{12}$.
\end{itemize}
As above, we ``reverse'' the arrows from \cite{mrw}, so as to better
generalise from the commutative situation.  Furthermore, for bicharacters,
we have translated ``to the left'', as we are working with left Haar weights,
and thus left multiplicative unitaries (which form the identity morphisms in
this category).

In ${\sf LCQG}$ we shall define the ``compact objects'' to be those $\G$ with
$C_0(\G)$ (or, equivalently, $C_0^u(\G)$) being unital.  Those $C_0(\G)$
thus arising are precisely the reduced compact quantum groups,
\cite[Section~2]{bmt}. Let ${\sf CQG}$ be the full subcategory of compact
quantum groups.  Henceforth, we shall write $C(\G)$ and $C^u(\G)$ to stress
that the algebra is unital.

To finish, notice that the relation between ${\sf LCQG}$ and ${\sf CSBa}$ is
slightly involved.  If we take the concrete realisation of ${\sf LCQG}$ as
having objects of the form $C_0^u(\G)$ and morphisms described by Hopf
$*$-homomorphisms, then ${\sf LCQG}$ becomes a full subcategory of
${\sf CSBa}$, and the compact objects agree.  However, this viewpoint is
slightly misleading, as for example $C_0(\G)$ will be an object in
${\sf CSBa}$ different from $C_0^u(\G)$, whereas we would generally regard
these as being ``the same'' quantum group.  Furthermore, of course,
${\sf CSBa}$ contains a great many objects which don't arise from ${\sf LCQG}$
in any fashion.  Of interest from the viewpoint of Section~\ref{sec:cat}
is that the ``compact'' objects do correspond, all be it in a many-to-one
fashion.

\subsection{So\l tan's Bohr compactification}

In \cite{soltan}, So\l tan showed that in ${\sf CSBa}$, compactifications
always exist.  We shall shortly give a full account of his theory, but for
now let us make some brief comments.  Given an object $\mc G=(A,\Delta_A)$ in
${\sf CSBa}$, we construct a certain unital C$^*$-subalgebra
$\mathbb{AP}(\mc G)$ in $M(A)$ such that (the strict extension of) $\Delta_A$
restricts to a coproduct $\Delta_{\mathbb{AP}(\mc G)}$ on $\mathbb{AP}(\mc G)$.
Then $\mf b\mc G = (\mathbb{AP}(\mc G),\Delta_{\mathbb{AP}(\mc G)})$ is the
compactification of $\mc G$ (termed the \emph{quantum Bohr compactification}
of $\mc G$ in \cite[Definition~2.14]{soltan}).

Given the concrete realisation of ${\sf LCQG}$ as a full subcategory of
${\sf CSBa}$, given $\G$ we can form $\mathbb{AP}(\G)$ by applying So\l tan's
theory to $C_0^u(\G)$.  However, the obvious problem is that
$(\mathbb{AP}(\G),\Delta_{\mathbb{AP}(\G)})$, while a compact
quantum group, might fail to be a \emph{universal} quantum group, and hence
would not be a member of ${\sf LCQG}$ (viewed as a subcategory of
${\sf CSBa}$).  Indeed, this can even occur when
$\G$ is cocommutative, see Section~\ref{sec:cocomm_red} below.

In the next section, we instead show how to adapt So\l tan's ideas to
construct a compactification in ${\sf LCQG}$.  We will also show that while
the resulting C$^*$-algebra picture is slightly complicated, the underlying
(unique) dense Hopf $*$-algebra can be constructed in a number of equivalent
ways, starting from either $C_0^u(\G)$ or from $C_0(\G)$.

\section{Compactification in ${\sf LCQG}$}
\label{sec:comp_lcqg}

In this section, we shall show how to construct compactifications in
${\sf LCQG}$, by somewhat directly applying So\l tan's construction.
Thus we first summarise So\l tan's work in \cite{soltan}.

Let $\mc G=(A,\Delta)$ be an object of ${\sf CSBa}$.  A (finite-dimensional)
\emph{bounded representation} of $\mc G$ is an element $T$ of 
$M(A) \otimes \mc B(H)$, where $H$ is a finite-dimensional Hilbert space,
with $(\Delta\otimes\iota)T = T_{12} T_{13}$, and such that $T$ is invertible.
Equivalently, we could
term such an object a \emph{invertible corepresentation} of $A$ (with the
$\Delta$ being clear from context) and we shall mostly stick to this
latter convention (to avoid confusing $C_0(\G)$ with $C_0^u(\G)$, and because
we want to stress the ``invertible'' aspect).
There are obvious notions of taking the direct sum, and tensor product,
of invertible corepresentations.

If we take a basis of $\mc B(H)$, then we establish an isomorphism
$\mc B(H) \cong \mathbb M_n$, and thus identify $T$ with $(T_{ij}) \in
\mathbb M_n(M(A))$.  Then $T$ needs to be invertible, and to satisfy
\[ \Delta(T_{ij}) = \sum_{k=1}^n T_{ik} \otimes T_{kj}
\quad\text{for all $i,j$.} \]
Let $T^\top = (T_{ji})$ be the ``transpose'' of $T$; this is always an
``anti-corepresentation''.  We shall say that $T$ is \emph{admissible} if
$T^\top$ is invertible.

\begin{remark}
If $A$ is commutative, then $T$ is invertible if and only if $T^\top$ is,
but in general the admissible corepresentations form a strict subset of the
collection of invertible corepresentations.  In ${\sf CSBa}$ there are
counter-examples (\cite[Remark~2.10]{soltan} which references
\cite[Section~4]{wang}) but we do not know the answer for
$C_0(\G)$ (and/or $C_0^u(\G)$) for $\G\in{\sf LCQG}$; see
Conjecture~\ref{conj:two} at the end.
\end{remark}

The linear span of the elements $T_{ij}$ form the collection of
\emph{matrix elements} of $T$.  By taking direct sums and tensor products,
and using the trivial representation, one can show that the collection of
matrix elements of invertible corepresentations forms a unital
subalgebra of $M(A)$, see \cite[Proposition~2.5]{soltan}.

Harder to show (as we now use Woronowicz's work in \cite{woro1}) is that
when $T$ is an admissible corepresentation, and $B_T$ denotes the C$^*$-algebra
generated by the matrix elements of $T$, then $\Delta$ restricts to
a map $B_T\rightarrow B_T\otimes B_T$, and $(B_T, \Delta|_{B_T})$ is
a compact quantum group, see \cite[Proposition~2.7]{soltan}.  It follows,
see \cite[Corollary~2.9]{soltan}, that $T$ is similar to a unitary
corepresentation (again, in ${\sf CSBa}$ the converse is not true, see
\cite[Remark~2.10]{soltan}).  One can now show that admissible
corepresentations are stable under tensor product, and it follows that the
set of all matrix elements of admissible corepresentations of $A$,
say $\mc{AP}(A)$, forms a unital $*$-subalgebra of $M(A)$,
see \cite[Proposition~2.12]{soltan}.

\begin{remark}\label{rem:uni_add_rep}
Let $T$ be a bounded corepresentation which is similar to a unitary
representation $U$.  Then it is elementary to see that $T^\top$ is invertible
if and only if $U^\top$ is invertible.  It follows that if we are
interested in the matrix elements of admissible corepresentations, it is
no loss of generality to study only admissible, \emph{unitary}
corepresentations.
\end{remark}

We shall again abuse notation slightly  (the map $\Delta$ being implicit)
and write $\mathbb{AP}(A)$ for the closure of $\mc{AP}(A)$.
Thus $\mathbb{AP}(A)$ is a
unital C$^*$-algebra and $\Delta$ restricts to $\mathbb{AP}(A)$ to give a
compact quantum group $\mf b A = (\mathbb{AP}(A), \Delta_{\mathbb{AP}(A)})$.
By \cite[Appendix~A]{bmt} we know that a compact quantum group admits a
unique dense Hopf $*$-algebra.
By combining this with \cite[Corollary~2.15]{soltan} we see that for
$\mathbb{AP}(A)$, this Hopf $*$-algebra is simply $\mc{AP}(A)$.
Finally, \cite[Theorem~3.1]{soltan} shows that $\mathbb{AP}(A)$ satisfies the
correct universal property to be, in the sense of Section~\ref{sec:cat},
the compactification of $(A,\Delta)$ in ${\sf CSBa}$.

\begin{remark}\label{rem:why_univ_prop}
Let $(A,\Delta_A)$ and $(B,\Delta_B)$ be C$^*$-bialgebras and $\theta:A
\rightarrow M(B)$ a Hopf $*$-homomorphism.  If $U\in M(A)\otimes \mathbb M_n$
is an admissible corepresentation, then $V=(\theta\otimes\iota)(U)$ will also
be admissible (as, for example, $\overline{V}=(\theta\otimes\iota)(\overline U)$
will have inverse $(\theta\otimes\iota)(\overline{U}^{-1})$).  It follows that
$\theta(\mathcal{AP}(A)) \subseteq \mathcal{AP}(B)$, and it is from this
observation that we see that $\mathbb{AP}(A)$ has the correct universal
property.
\end{remark}

\subsection{For locally compact quantum groups}

Let $\G$ be an object of ${\sf LCQG}$.  Apply So\l tan's theory to the
universal quantum group $C_0^u(\G)$, to yield a compact quantum group
$\mathbb{AP}(C_0^u(\G))$.  This defines an object $\mathbb K$ of
${\sf LCQG}$ which is compact.  Indeed, $C(\K)$ is the \emph{reduced version}
of $\mathbb{AP}(C_0^u(\G))$, see \cite[Theorem~2.1]{bmt}.  This can be
formed as the quotient of $\mathbb{AP}(C_0^u(\G))$ by the null-ideal of the
Haar state.  Alternatively, we can start with the Hopf $*$-algebra
$\mc{AP}(C_0^u(\G))$, which also carries a Haar state.
Using this we can form the Hilbert space $L^2(\K)$, and then we can identify
$\mc{AP}(C_0^u(\G))$ with a $*$-algebra of operators on $L^2(\K)$.
Then $C(\K)$ is the closure of $\mc{AP}(C_0^u(\G))$ in $\mc B(L^2(\K))$.
See \cite[Theorem~2.11]{bmt}.

We can now form $C^u(\K)$, the universal version of $\K$, by following
\cite{kus1}.  Alternatively, one can follow \cite[Section~3]{bmt}; a little
work shows that these are equivalent constructions.  In particular, we can
think of $C^u(\K)$ as being the universal enveloping C$^*$-algebra of
$\mc{AP}(C_0^u(\G))$, and so there is a surjective $*$-homomorphism
$\Lambda^u_{\mathbb{AP}} : C^u(\K) \rightarrow \mathbb{AP}(C_0^u(\G))$
which intertwines the coproducts.  The composition
\[ \xymatrix{ C^u(\K) \ar[r]^-{\Lambda^u_{\mathbb{AP}}}
& \mathbb{AP}(C_0^u(\G)) \ar@{^{(}->}[r]
& M(C_0^u(\G)) } \]
is hence a Hopf $*$-homomorphism, and so defines an arrow $\G \rightarrow
\K$ in ${\sf LCQG}$.  As we might hope, we have the following result:

\begin{proposition}\label{prop:compact_lcqg}
With the arrow $\G\rightarrow\K$ just defined, $\K$ is the compactification
of $\G$ in ${\sf LCQG}$.
\end{proposition}

Before we can prove this result, we need to study in detail the ideas of
\cite{mrw} and \cite{kus1}, as applied to compact quantum groups.

\subsection{Morphisms and lifts}\label{sec:mor_lifts}

Let us recall some notions related to one and two-sided ``universal
bicharacters'' (to use the language of \cite{mrw}).  We shall follow
\cite{kus1}, but analogous results are shown in \cite{mrw,sw}.  Let $\G$ be
a compact quantum group and form the ``universal'' algebra $C_0^u(\G)$
together with the reducing morphism $\Lambda_\G:C_0^u(\G)\rightarrow
C_0(\G)$ (which is denoted by $\pi$ in \cite{kus1}).
There is a unitary $\mc V_{\G} = \mc V \in M(C_0^u(\G)
\otimes C_0(\hat\G))$ such that $(\Delta_u\otimes\iota)(\mc V)
= \mc V_{13}\mc V_{23}$ and $(\iota\otimes\hat\Delta)(\mc V) = \mc V_{13}
\mc V_{12}$; we think of $\mc V$ as being a variant of $W$ with its
left-leg in $C_0^u(\G)$; indeed, $(\Lambda_\G\otimes\iota)(\mc V)=W$.
Similarly, there is $\mc U\in M(C_0^u(\G)\otimes C_0^u(\hat\G))$, a 
fully universal version of $W$.

We now recall some results from \cite{mrw}.
Let $(A,\Delta)$ be a C$^*$-bialgebra, and let $\G$ be a locally compact
quantum group.  We shall say that $U\in M(A\otimes C_0(\hat\G))$ is a
\emph{bicharacter} if $U$ is unitary, $(\Delta\otimes\iota)(U) = U_{13}U_{23}$
and $(\iota\otimes\Delta_{\hat\G})(U) = U_{13}U_{12}$.  Then
\cite[Proposition~4.2]{mrw} shows that there is a bijection between
such bicharacters $U$, and Hopf $*$-homomorphisms $\phi:C_0^u(\G)\rightarrow
A$, the link being that $U = (\phi\otimes\iota)(\mc V_{\G})$.

For locally compact quantum groups $\G,\H$, we can similarly define the
notion of a bicharacter in $M(C_0^u(\G)\otimes C_0^u(\hat\H))$.  Then
\cite[Proposition~4.7]{mrw}
shows that for any bicharacter $U\in M(C_0(\G)\otimes C_0(\hat\H))$, there
is a unique bicharacter $V \in M(C_0^u(\G)\otimes C_0^u(\hat\H))$ with
$(\Lambda_{\G} \otimes \Lambda_{\hat\H})(V) = U$.  We call $V$ the ``lift''
of $U$; hence $\mc U$ is the lift of $W$.

The following is only implicit in \cite{mrw} (having been rather more explicit
in an early preprint\footnote{See \texttt{http://arxiv.org/abs/1011.4284v1}}
of that paper) so we give the short argument.

\begin{proposition}\label{prop:uni_lifts}
Given locally compact quantum groups $\G,\H$ and a Hopf $*$-homomorphism
$\theta:C_0^u(\G) \rightarrow C_0(\H)$, there is a unique
Hopf $*$-homomorphism $\theta_0:C_0^u(\G) \rightarrow C_0^u(\H)$ with
$\theta = \Lambda_{\H} \theta_0$.
\end{proposition}
\begin{proof}
Let $U \in M(C_0(\H) \otimes C_0(\hat\G))$ be the unique bicharacter
associated with $\theta$, and let $V \in M(C_0^u(\H) \otimes C_0^u(\hat\G))$
be the unique ``lift''.  Then $(\iota\otimes\Lambda_{\hat\G})(V)$ is
a bicharacter, and so gives a unique $\theta_0:C_0^u(\G) \rightarrow C_0^u(\H)$.
Then observe that $(\iota\otimes\Lambda_{\hat\G})(V) =
(\theta_0\otimes\iota)(\mc V_{\G})$.  It follows that
\[ (\Lambda_{\H} \theta_0\otimes\iota)(\mc V_{\G})
= (\Lambda_{\H}\otimes\Lambda_{\hat\G})(V) = U
= (\theta\otimes\iota)(\mc V_{\G}). \]
As $\{ (\iota\otimes\omega)(\mc V_{\G}) : \omega\in L^1(\hat\G) \}$
is dense in $C_0^u(\G)$, it follows that $\Lambda_{\H} \theta_0 = \theta$.
Uniqueness follows in a similar way to \cite[Lemma~50]{mrw}, 
\cite[Lemma~6.1]{kus1}, compare Lemma~\ref{lem:uni_lift_cqg} below.
\end{proof}

Now let $A$ be a compact quantum group, not assumed to be
reduced or universal.  Let $\K$ be the abstract compact quantum group
determined by $A$, so that $C(\K)$ is the reduced version of $A$, and
$C^u(\K)$ is the universal version of $A$.  As above, if $\mc A$ is the
unique dense Hopf $*$-algebra of $A$, then $C(\K)$ is the completion of
$\mc A$ determined by the Haar state, and $C^u(\K)$ is the universal
C$^*$-algebra completion of $\mc A$.  Let $\Lambda^u_A:C^u(\K) \rightarrow
A$ and $\Lambda^r_A:A\rightarrow C(\K)$ be the surjective Hopf
$*$-homomorphisms which make the following diagram commute:
\[ \xymatrix{ C^u(\K) \ar@{->>}[r]^-{\Lambda^u_A} &
A \ar@{->>}[r]^-{\Lambda^r_A} & C(\K) \\
& \mc A \ar@{^{(}->}[lu] \ar@{^{(}->}[u] \ar@{^{(}->}[ru] } \]
We also have the surjective Hopf $*$-homomorphism $\Lambda_{\K}:
C^u(\K) \rightarrow C(\K)$.  As this also respects the inclusion of
$\mc A$ into $C_0^u(\K)$ and $C_0(\K)$, it follows that we have a further
commutative diagram:
\[ \xymatrix{ C^u(\K) \ar[r]^-{\Lambda^u_A}
\ar@/_1pc/[rr]_-{\Lambda_\K}
& A \ar[r]^-{\Lambda^r_A} & C(\K) } \]

The following lemma uses similar techniques to \cite[Result~6.1]{kus1}.

\begin{lemma}\label{lem:uni_lift_cqg}
Let $(A,\Delta_A)$ and $\Lambda^r_A$ be as above.  Let $\H$ be a locally
compact quantum group.  If $\pi_1,\pi_2:C_0^u(\H) \rightarrow A$ are Hopf
$*$-homomorphisms such that $\Lambda^r_A\pi_1 = \Lambda^r_A\pi_2$, then
$\pi_1=\pi_2$.
\end{lemma}
\begin{proof}
Consider the universal left regular representation $\mc V$ for $\K$.  By
\cite[Proposition~6.2]{kus1} we have that
\[ (\iota\otimes\Lambda_\K)\Delta^u_\K(x)
= \mc V^*(1\otimes\Lambda_\K(x))\mc V   \qquad (x\in C^u(\K)). \]
Set $U = (\Lambda^u_A\otimes\iota)(\mc V) \in M(A \otimes C_0(\hat\K))$,
so that
\[ U^*(1\otimes\Lambda_\K(x)) U
= (\Lambda^u_A \otimes\Lambda_\K)\Delta^u_\K(x)
= (\iota\otimes\Lambda^r_A)\Delta_A(\Lambda^u_A(x))
\qquad (x\in C^u(\K)). \]
It follows that
\[ U^*(1\otimes \Lambda^r_A(x)) U = (\iota\otimes\Lambda^r_A)\Delta_A(x)
\qquad (x\in A). \]
We remark that we could construct $U$ purely using compact quantum group
techniques, compare equation (5.10) in \cite{woro2} (and remember that we
work with left multiplicative unitaries).

Now let $\pi:C_0^u(\H)\rightarrow A$ be a Hopf $*$-homomorphism.
Let $\mc U$ be the universal bicharacter of $\H$
and set $V = (\pi\otimes\iota)(\mc U) \in M(A\otimes C_0^u(\hat\H))$.
Then
\[ (\Delta_A\otimes\iota)(V) = (\pi\otimes\pi\otimes\iota)
(\Delta^u_\H\otimes\iota)(\mc U)
= V_{13} V_{23}. \]
By combining the previous two displayed equations, we see that
\[ V_{13} ((\Lambda^r_A\otimes\iota)(V))_{23}
= ((\iota\otimes\Lambda^r_A)\Delta_A \otimes \iota)(V)
= U^*_{12} ((\Lambda^r_A\otimes\iota)(V))_{23} U_{12}, \]
and so
\[ V_{13} = U^*_{12} ((\Lambda^r_A\otimes\iota)(V))_{23} U_{12}
((\Lambda^r_A\otimes\iota)(V))_{23}^*. \]
It follows that $V$ is determined by $(\Lambda^r_A\otimes\iota)(V)
= (\Lambda^r_A\pi\otimes\iota)(\mc U)$, that is, $V$ is determined
by $\Lambda^r_A\pi$.
By \cite[Corollary~6.1]{kus1}, $(\iota\otimes\Lambda_{\hat\H})(\mc U)
= \mc V_\H$, and by the remarks after \cite[Proposition~5.1]{kus1}, we
know that $\{ (\iota\otimes\omega)(\mc V_\H) : \omega\in L^1(\hat\H) \}$
is dense in $C_0^u(\H)$.  Thus $\pi$ is determined by knowing
\[ \pi( (\iota\otimes\omega)(\mc V_\H) ) = 
\pi( (\iota\otimes\omega \Lambda_{\hat\H})(\mc U) )
= (\iota\otimes\omega \Lambda_{\hat\H})(V). \]
We conclude that $\pi$ is determined uniquely by knowing $V$, and
in turn $V$ is uniquely determined by knowing $\Lambda^r_A \pi$.
The result follows.
\end{proof}

\subsection{Back to compactifications}

We are now in a position to prove Proposition~\ref{prop:compact_lcqg}.
Let $\K$ be the compact quantum defined by $\mathbb{AP}(C_0^u(\G))$.
Let $\Lambda^\K_{\mathbb{AP}}:C^u(\K)\rightarrow M(C_0^u(\G))$ be the
Hopf $*$-homomorphism defining the arrow $\G\rightarrow\K$, and let
$\Lambda^r_{\mathbb{AP}}:\mathbb{AP}(C_0^u(\G)) \rightarrow C(K)$ be
the ``reducing morphism'' considered in the previous section.

\begin{proof}[Proof of Proposition~\ref{prop:compact_lcqg}]
Let $\H$ be compact in ${\sf LCQG}$, and let $\G\rightarrow\H$ be an arrow.
We have to show that this factors through $\G\rightarrow\K$.
Let the arrow $\G\rightarrow\H$ correspond
to the Hopf $*$-homomorphism $\theta:C^u(\H) \rightarrow C_0^u(\G)$.
As $\mathbb{AP}(C_0^u(\G))$ is a compactification, compare
Remark~\ref{rem:why_univ_prop}, it follows that $\theta(C^u(\H)) \subseteq
\mathbb{AP}(C_0^u(\G)) \subseteq M(C_0^u(\G))$.  So the composition
\[ \xymatrix{ C^u(\H) \ar[r]^-{\theta} &
\mathbb{AP}(C_0^u(\G)) \ar[r]^-{\Lambda^r_{\mathbb{AP}}} &
C(\K) } \]
makes sense, and is a Hopf $*$-homomorphism $\theta_0:C^u(\H) \rightarrow
C(\K)$.  Let $\theta_1:C^u(\H)\rightarrow C^u(\K)$ be the
unique lift given by Proposition~\ref{prop:uni_lifts}.  This defines
an arrow $\K\rightarrow\H$ in ${\sf LCQG}$.

We now claim that the following diagrams commute:
\[ \xymatrix{ \G \ar[r] \ar[rd] & \H \\
& \K \ar[u] } \qquad
\xymatrix{ C_0^u(\G) & C^u(\H) \ar[l]_-{\theta} \ar[d]^-{\theta_1} \\
& C^u(\K) \ar[lu]^-{\Lambda^{\K}_{\mathbb{AP}}} } \]
By the definition of arrows in ${\sf LCQG}$, one diagram commutes if and only
if the other does.  However, we calculate that
\[ \Lambda^r_{\mathbb{AP}} \theta = \theta_0
= \Lambda_{\K} \theta_1
= \Lambda^r_{\mathbb{AP}} \Lambda^{\K}_{\mathbb{AP}} \theta_1. \]
By Lemma~\ref{lem:uni_lift_cqg}, it follows that
$\theta = \Lambda^{\K}_{\mathbb{AP}} \theta_1$, as required.
\end{proof}

\begin{definition}
Given $\G$, let the resulting compact quantum group $\K$ be denoted by
$\G^{\sap}$, the \emph{strongly almost periodic} compactification of $\G$.
\end{definition}

One could equally well call this the ``Bohr compactification'' of $\G$, but
this terminology would clash with that used by So\l tan in \cite{soltan}
(because $\G^\sap$ is an abstract quantum group, not in general a concrete
sub-C$^*$-bialgebra of $M(C_0(\G))$).
Our terminology is inspired by that for semigroups, see
\cite[Section~4.3]{berg} (for ``reasonable'' semigroups, the ``strongly almost
periodic compactification'' is the universal compact group compactification,
while the ``almost periodic compactification'' is the universal compact
\emph{semigroup} compactification.  For topological groups, the notions
coincide: it would be interesting to investigate analogous ideas for
C$^*$-bialgebras).

\begin{remark}\label{rem:what_is_cmpt_functor}
Recall from Section~\ref{sec:cat} that for an arrow $\G\rightarrow\H$ in
${\sf LCQG}$ we have a unique arrow $\G^\sap \rightarrow \H^\sap$.
If $\G\rightarrow\H$ is given by a Hopf $*$-homomorphism $\theta:
C_0^u(\H)\rightarrow C_0^u(\G)$, and $\G^\sap \rightarrow \H^\sap$ is
given by $\theta^\sap:C^u(\H^\sap) \rightarrow C^u(\G^\sap)$, then by
construction, this is the unique Hopf $*$-homomorphism making the following
diagram commute:
\[ \xymatrix{ C^u(\H^\sap) \ar@{->>}[r]^{\Lambda^u_{\mathbb{AP}}}
\ar[d]^-{\theta^\sap} &
\mathbb{AP}(C_0^u(\H)) \ar@{^{(}->}[r] &
M(C_0^u(\H)) \ar[d]^-{\theta} \\
C^u(\G^\sap) \ar@{->>}[r]^{\Lambda^u_{\mathbb{AP}}} &
\mathbb{AP}(C_0^u(\G)) \ar@{^{(}->}[r] &
M(C_0^u(\G)) } \]
Now, the composition $\mathbb{AP}(C_0^u(\H)) \rightarrow
M(C_0^u(\H)) \overset{\theta}{\rightarrow} M(C_0^u(\G))$ is a
Hopf $*$-homomorphism, and so,
again by \cite[Theorem~3.1]{soltan}, it follows that the image is a subset
of $\mathbb{AP}(C_0^u(\G))$.  So actually $\theta^\sap$ drops to a
Hopf $*$-homomorphism $\mathbb{AP}(C_0^u(\H)) \rightarrow
\mathbb{AP}(C_0^u(\G))$; that is, provides an arrow in the middle
vertical in the diagram above.

Indeed, by Remark~\ref{rem:why_univ_prop} we know
that $\theta(\mathcal{AP}(C_0^u(\H))) \subseteq \mathcal{AP}(C_0^u(\G))$,
and so $\theta$ restricts to give a map $\mathbb{AP}(C_0^u(\H)) \rightarrow
\mathbb{AP}(C_0^u(\G))$; this is the map considered in the previous paragraph.
By composing with the inclusion $\mathcal{AP}(C_0^u(\G))\rightarrow
C^u(\G^\sap)$ we obtain a Hopf $*$-homomorphism $\mathcal{AP}(C_0^u(\H))
\rightarrow C^u(\G^\sap)$.  As $C^u(\H^\sap)$ is the universal C$^*$-algebra
generated by $\mathcal{AP}(C_0^u(\H))$, we hence obtain a map
$C^u(\H^\sap) \rightarrow C^u(\G^\sap)$, and by tracing the construction in
the proof of Proposition~\ref{prop:compact_lcqg}, we find that this map is
indeed $\theta^\sap$.
\end{remark}

We next investigate what would happen if we used $\mathbb{AP}(C_0(\G))$
instead of $\mathbb{AP}(C_0^u(\G))$.

\begin{proposition}\label{prop:same_red_uni}
Let $\G$ be a locally compact quantum group.  Consider the Hopf $*$-algebras
$\mc{AP}(C_0(\G))$ and $\mc{AP}(C_0^u(\G))$, which we can consider as
subalgebras of $M(C_0^u(\G))$ and $M(C_0(\G))$, respectively.
Then the strict extension of
$\Lambda:C_0^u(\G) \rightarrow C_0(\G)$ restricts to form a bijection
$\mc{AP}(C_0^u(\G)) \rightarrow \mc{AP}(C_0(\G))$.  In particular,
$C(\G^\sap)$ is also the reduced version of the compact
quantum group $\mathbb{AP}(C_0(\G))$.
\end{proposition}
\begin{proof}
By construction (see after Definition~\ref{rem:uni_add_rep})
$\mc{AP}(C_0(\G))$ is merely the set of elements of admissible
corepresentations of $C_0(\G)$; and 
similarly for $\mc{AP}(C_0^u(\G))$.  As $\Lambda$ is a Hopf-$*$-homomorphism,
it is clear that $\Lambda(\mc{AP}(C_0^u(\G))) \subseteq \mc{AP}(C_0(\G))$.

Conversely, and with reference to Remark~\ref{rem:uni_add_rep}, let
$U_0$ be an admissible unitary corepresentation of $C_0(\G)$, and let $U$ be
the unique lift to a unitary corepresentation of $C_0^u(\G)$.  Our aim is to
show that $U$ is admissible, from which it will follow that 
$\Lambda(\mc{AP}(C_0^u(\G))) = \mc{AP}(C_0(\G))$.

It is easy to see that $U_0^T$ is invertible if and only if $\overline{U_0}
= (U_{0,ij}^*)_{i,j=1}^n$ is invertible.  Now, $\overline{U_0}$ is a
corepresentation (and not an anti-corepresentation), and from the theory of
compact quantum (matrix) groups
we know that $\overline{U_0}$ is similar to a unitary, so there is a scalar
matrix $F$ with $V_0=F^{-1} \overline{U_0} F$ unitary.  Let $V$ be the 
unique lift to $C_0^u(\G)$.  We shall show that $\overline{U} =
F V F^{-1}$, which is invertible, showing that $U$ is admissible as claimed.

We now argue as in the proof of Lemma~\ref{lem:uni_lift_cqg}.
For $i,j$, we have that
\begin{align*} \sum_k U_{ik} \otimes U_{0,kj}
&= (\iota\otimes\Lambda)\Delta_u(U_{ij})
= \mc V^*(1\otimes U_{0,ij})\mc V
= \mc V^*(1\otimes (F V_0 F^{-1})_{ij}^*)\mc V \\
&= \sum_{s,t} \overline{F_{is}} \overline{F^{-1}_{tj}}
   \mc V^*(1\otimes V_{0,st}^*) \mc V
= \sum_{s,t} \overline{F_{is}} \overline{F^{-1}_{tj}}
   (\iota\otimes\Lambda)\Delta_u(V_{st}^*) \\
&= \sum_{s,t,k} \overline{F_{is}} \overline{F^{-1}_{tj}}
   V_{sk}^* \otimes V_{0,kt}^*
= \sum_k (FV)_{ik}^* \otimes (V_0 F^{-1})_{kj}^* \\
&= \sum_k (FV)_{ik}^* \otimes (F^{-1}\overline{U_0})_{kj}^*
= \sum_k (FV)_{ik}^* \otimes (\overline{F^{-1}} U_0)_{kj}.
\end{align*}
It then follows that as $U_0$ is unitary, for each $i,r$,
\begin{align*} \sum_{k,j} U_{ik} \otimes U_{0,kj} U_{0,rj}^*
&= \sum_k U_{ik} \otimes \delta_{k,r} 1 = U_{ir} \otimes 1 \\
&= \sum_{k,j} (FV)_{ik}^* \otimes (\overline{F^{-1}} U_0)_{kj}
  (U_{0}^*)_{jr}
= \sum_k (FV)_{ik}^* \otimes \overline{F^{-1}}_{kr}
= (FVF^{-1})^*_{ir} \otimes 1. \end{align*}
Hence $\overline{U} = FVF^{-1}$ as claimed.

So we have shown that admissible unitary corepresentations of $C_0(\G)$
lift to admissible unitary corepresentations of $C_0^u(\G)$.
Finally, we argue as in \cite[Section~1.2]{dkss}.
Let $\mc B_u\subseteq M(C_0^u(\G))$ denote the space of elements
of all unitary corepresentations of $C_0^u(\G)$.  By the universal
property of $\mc U$, it follows that
\[ \mc B_u = \{ (\iota\otimes\omega\circ\phi)(\mc U) :
\phi:C_0^u(\hat\G) \rightarrow \mc B(H) \text{ is a non-degenerate
$*$-homomorphism }, \omega\in\mc B(H)_* \}. \]
Similarly define $\mc B\subseteq M(C_0(\G))$, so
\[ \mc B = \{ (\iota\otimes\omega\circ\phi)(\hat{\mc V}) :
\phi:C_0^u(\hat\G) \rightarrow \mc B(H) \text{ is a non-degenerate
$*$-homomorphism }, \omega\in\mc B(H)_* \}. \]
Then $\Lambda$ restricts to a surjection $\mc B_u\rightarrow\mc B$,
because $(\Lambda\otimes\iota)(\mc U) = \hat{\mc V}$.  We claim that
$\Lambda:\mc B_u\rightarrow\mc B$ is an injection, from which it will follow
certainly that $\Lambda : \mc{AP}(C_0^u(\G)) \rightarrow \mc{AP}(C_0(\G))$
is injective, as required.

Let $a=(\iota\otimes\omega\circ\phi)(\mc U) \in \mc B_u$ be non-zero.
So $\mu = \omega\circ\phi\in C_0^u(\hat\G)^*$ is non-zero.  Now,
$C_0^u(\hat\G)$ is the closed linear span
of $\{ (\tau\otimes\iota)(\hat{\mc V}) : \tau\in L^1(\G) \}$, and as
$\mu\not=0$, there is $\tau\in L^1(\G)$ with
$\ip{\mu}{(\tau\otimes\iota)(\hat{\mc V})} \not=0$.  Thus
$\Lambda(a) = (\iota\otimes\mu)(\hat{\mc V}) \not =0$, as required.

As $C(\G^\sap)$ is the completion of $\mc{AP}(C_0^u(\G))$ for the norm coming
from the action on $L^2(\G^\sap)$, it follows that $C(\G^\sap)$ is also the
completion of $\mc{AP}(C_0(\G))$, namely the reduced version of
$\mathbb{AP}(C_0(\G))$, and so the second claim follows.
\end{proof}

Consequently, it is enough to work in $C_0(\G)$.  Combining this proposition
with the observations of Section~\ref{sec:mor_lifts}, we have the following
commutative diagram, where now $\mc A$ denotes the Hopf $*$-algebra
associated with $\G^\sap$, which can now be identified the space of elements
of admissible representations of $C_0^u(\G)$, or equivalently, $C_0(\G)$,
\[ \xymatrix{ C^u(\G^\sap) \ar@{->>}[r] &
\mathbb{AP}(C_0^u(\G)) \ar@{->>}[rr]^-{\text{Restriction of }\Lambda_{\G}} &&
\mathbb{AP}(C_0(\G)) \ar@{->>}[r] &
C(\G^\sap) \\
&& \mc A \ar[ull] \ar[ul] \ar[ur] \ar[urr]
} \]
In Section~\ref{sec:cocomm_red} below, we shall see that in the cocommutative
case, it is possible to say when the horizontal surjections are actually
isomorphisms.

\begin{definition}\label{defn:curlyap_of_g}
For a locally compact quantum group $\G$, we write $\mc{AP}(\G)$ for the
unique dense Hopf $*$-algebra of $C(\G^\sap)$.  By the above, equivalently
this is the unique dense Hopf $*$-algebra of $\mathbb{AP}(C_0(\G))$, or
of $\mathbb{AP}(C_0^u(\G))$.
\end{definition}

We finish this section by showing a simple link between admissible
representations and the antipode $S$, thought of here as
a strictly-closed (unbounded) operator on $M(C_0(\G))$.

\begin{proposition}\label{prop:admiss_iff_dsinv}
Let $U=(U_{ij})\in \mathbb M_n(M(C_0(\G)))$ be a unitary corepresentation.
Then $U$ is admissible if and only if $U_{ij}^* \in D(S)$ for each $i,j$.
\end{proposition}
\begin{proof}
For fixed $i,j$, as $U$ is a corepresentation
and is unitary,
\begin{align*} \sum_k \Delta(U_{ik}) (1\otimes U^*_{jk})
= \sum_{k,l} U_{il} \otimes U_{lk} U^*_{jk}
= \sum_l U_{il} \otimes (UU^*)_{lj} = U_{ij} \otimes 1.
\end{align*}
Similarly,
\begin{align*} \sum_k (1\otimes U_{ik}) \Delta(U^*_{jk})
= \sum_{k,l} U^*_{jl} \otimes U_{ik} U^*_{lk}
= \sum_l U^*_{jl} \otimes (UU^*)_{il} = U^*_{ji} \otimes 1. \end{align*}
It follows from \cite[Corollary~5.34, Remark~5.44]{kv} that $U_{ij} \in D(S)$
with $S(U_{ij}) = U_{ji}^*$.

Suppose now that $U_{ij}^*\in D(S)$ for all $i,j$.  Then
$U_{ij}\in D(S^{-1})$, so we may set $V_{ij} = S^{-1}(U_{ji})$.  Then
\[ \sum_k V_{ik} (U^\top)_{kj} = \sum_k S^{-1}(U_{ki}) U_{jk}
= S^{-1}\Big( \sum_k S(U_{jk}) U_{ki} \Big)
= S^{-1}\Big( \sum_k U_{kj}^* U_{ki} \Big) = \delta_{i,j} 1, \]
so that $VU^\top = 1$.  Similarly, $U^\top V=1$, so $U^\top$ is invertible.

Conversely, if $U$ is admissible, then the elements of $U$ will belong
to the Hopf $*$-algebra $\mc A$ associated to $\mathbb{AP}(C_0(\G))$.
By applying \cite[Proposition~5.45]{kv} (compare \cite[Proposition~4.11]{soltan})
to the inclusion Hopf $*$-homomorphism $\mathbb{AP}(C_0(\G)) \rightarrow
M(C_0(\G))$ we see that this inclusion will intertwine the antipode on
$\mc A$ and $S$.  In particular, $S$ will restrict to a bijection
$\mc A\rightarrow\mc A$, and so the ``only if'' claim follows.
\end{proof}

\begin{remark}
Thanks to \cite[Proposition~9.6]{kus1}, the same result holds for unitary
corepresentations of $C_0^u(\G)$ if we use the universal antipode $S_u$.

We also remark that the proof that matrix elements of unitary
corepresentations of a compact quantum group form a Hopf $*$-algebra
ultimately relies upon the fact that if $U$ is a unitary corepresentation,
then $\overline{U}$, or equivalently $U^\top$, is similar to a unitary
corepresentation (equivalently, anti-corepresentation).  Indeed, we
used this fact in the proof of Proposition~\ref{prop:same_red_uni} above.
This point was, we feel, slightly skipped in \cite[Remark~2.3(2)]{soltan}
and explains why the argument given there does not work (directly) for
\emph{locally} compact quantum groups.
\end{remark}

\begin{remark}
If $\G$ is a Kac algebra (or, more generally, if $S=R$) then
the previous proposition gives a simple proof that any unitary corepresentation
is admissible.  In particular, this answers the implicit question before
\cite[Proposition~4.6]{soltan}.
\end{remark}

\begin{remark}
Let us just remark that everything in this section applies equally well to
quantum groups coming from manageable multiplicative unitaries-- simply
replace references to \cite{kus1} by the appropriate results to be found
in \cite{sw} and \cite{mrw}.
\end{remark}

\section{Representation free, and Banach algebraic, techniques}\label{sec:ba_sec}

When $G$ is a locally compact group, the algebra $\mathbb{AP}(C_0(G))$
coincides with the classical algebra of \emph{almost periodic functions},
namely those functions $f\in C^b(G)$ such that the collection of left
(or right) translates of $f$ forms a relatively compact subset of
$C^b(G)$, see \cite[Section~4.3]{berg} for example.

There is a classical and well-studied link with Banach algebras here.
Consider the algebra $L^1(G)$ and turn $C^b(G)$ into an $L^1(G)$ bimodule
in the usual way; we shall denote the module actions by $\star$.  Then
a simple argument using the bounded approximate identity for $L^1(G)$,
together with the fact that a subset of a Banach space is relatively
compact if and only if its absolutely convex hull is, shows that
a function $f$ is almost periodic if and only if the orbit map $L^1(G)
\rightarrow C^b(G); a\mapsto a\star f$ (or $f\star a$) is a compact linear
map.  In fact, identifying $C^b(G)$ with a subalgebra of $L^\infty(G)$, we
obtain the same class by looking at those $f\in L^\infty(G)$ with
$L^1(G)\rightarrow L^\infty(G); a\mapsto a\star f$ being compact
(compare with the arguments of \cite{ulger} or \cite[Lemma~5.1]{chou}
for example).

We are hence lead to consider the following definition.

\begin{definition}
Let $\mf A$ be a Banach algebra, and turn $\mf A^*$ into an $\mf A$-bimodule
in the usual way.  A functional $\mu\in\mf A^*$ is \emph{almost periodic}
if the orbit map $\mf A\rightarrow\mf A^*; a\mapsto a\star \mu$ is
compact.  We write $AP(\mf A)$ for the collection of such functionals.
\end{definition}

For $\mu\in\mf A^*$, define $L_\mu,R_\mu:\mf A\rightarrow\mf A^*$ by
$R_\mu(a)=a\star\mu$ and $L_\mu(a)=\mu\star a$.  With $\kappa:\mf A\rightarrow
\mf A^{**}$ being the canonical map, we find that $L_\mu^*\circ\kappa =
R_\mu$ and $R_\mu^*\circ\kappa = L_\mu$, and so $L_\mu$ is compact if and only
if $R_\mu$ is compact.  We remark that some authors write $AP(\mf A^*)$ instead
of $AP(\mf A)$.

When $\mf A=L^1(G)$, we hence recover the classical notion of an almost
periodic function.  It is thus very tempting to use the same definition
for any locally compact quantum group.  Indeed, early on in the development
of the Fourier Algebra, the definition of $AP(\hat G) = AP(A(G))$ was
made (we believe for the first time in \cite[Chapter~7]{dr1}).

Let us quickly recall some notation.  For a locally compact group $G$
we define $L^\infty(\hat G)$ to be the group von Neumann algebra $VN(G)$,
which is generated by the left translation maps $\{\lambda(s):s\in G\}$ acting
on $L^2(G)$.  In the setup of Kac algebras or locally compact quantum groups,
this is the dual to the commutative algebra $L^\infty(G)$.  The predual of
$VN(G)$ is the Fourier algebra $A(G)$, as defined by Eymard, \cite{eymard}
(compare also \cite[]{tak2}).  Generally one thinks of $A(G)$ as being a
commutative Banach algebra, in fact, a subalgebra of $C_0(G)$.  We remark that
the locally compact quantum group convention would be to consider multiplicative
unitary $W$ for $VN(G)$, and then to consider the ``left-regular representation'',
the map $A(G)=VN(G)_*\rightarrow C_0(G); \omega\mapsto (\omega\otimes\iota)(W)$
(which is also used in \cite[Section~3, Chapter~VII]{tak2}).
Concretely, we identify the functional
$\omega$ with the continuous function $G\rightarrow\mathbb C; t\mapsto
\ip{\lambda(t^{-1})}{\omega}$.  We warn the reader that Eymard instead considers
the map $t\mapsto \ip{\lambda(t)}{\omega}$.

It was recognised that it ``should be'' the case that $AP(\hat G) =
C^*_\delta(G)$, the C$^*$-subalgebra of $VN(G)$ generated by the left
translation operators.  This is indeed quantum Bohr compactification,
see \cite[Section~4.2]{soltan} and Section~\ref{sec:cocomcase} below.
An excellent reference here is Chou's paper \cite{chou}.

\begin{theorem}\label{thm:chou}
We have the following facts:
\begin{enumerate}
\item $AP(\hat G) = C^*_\delta(G)$ if $G$ is abelian \cite{dr1},
or discrete and amenable (which follows easily from \cite[Proposition~2]{gr1}).
\item there exists a compact group $G$ such that $C^*_\delta(G)
\not= AP(\hat G)$.  This is \cite[Theorem~3.5]{chou}, but be aware of some
errors in preliminary results; these errors are partly corrected in
\cite{rindler}; in particular \cite[Proposition~1]{rindler} shows the result
we are interested in.
\item if $H$ is an open normal subgroup of $G$ with $G/H$ amenable, and
with $C^*_\delta(H) = AP(\hat H)$, then also $C^*_\delta(G) = AP(\hat G)$,
\cite[Theorem~4.4]{chou}.
\end{enumerate}
\end{theorem}

Remarkably, in full generality, it is still unknown if $AP(\hat G)$ is
even a C$^*$-algebra, never-mind whether it satisfies any obvious
interpretation as a ``compactification''.
Recently Runde suggested a new definition of ``almost periodic'' which
takes account of the Operator Space structure of $A(G)$.  We use standard
notions from the theory of Operator Spaces, see \cite{er} for example.
In particular, if $M$ is a von Neumann algebra and $M_*$ its predual, then
the space of completely bounded maps $M_*\rightarrow M$, denoted
$\mc{CB}(M_*,M)$, can be identified with the dual space of the operator
space projective tensor product, $M_* \proten M_*$, and with the von Neumann
tensor product $M\vnten M$.  That is, we identify $T\in \mc{CB}(M_*,M)$,
$\mu\in (M_* \proten M_*)^*$ and $y\in M\vnten M$ by the relations
\[ \ip{T(\omega)}{\tau} = \ip{\mu}{\omega\otimes\tau}
= \ip{y}{\omega\otimes\tau} \qquad (\omega,\tau\in M_*). \]
There are analogous constructions given by slicing $y$ on the right;
see \cite[Chapter~7]{er}.

For a map between operator spaces, there are a number of notions of being
``compact'', namely \emph{completely compact} and \emph{Gelfand compact}, see
\cite[Section~1]{runde} and references therein for a discussion.  In general
these are distinct, but when mapping into a dual, \emph{injective} operator
space, they coincide with the notation of being the completely-bounded-norm
limit of finite-rank operators (much as, in the presence of the approximation
property, a compact map between Banach spaces is the norm limit of finite-rank
maps).  For a completely contractive Banach algebra $\mf A$, Runde makes the
following definitions:

\begin{definition}
A completely bounded map $T:E\rightarrow F$ between operator spaces is
\emph{completely compact} if for each $\epsilon>0$ there is a
finite-dimensional subspace $Y$ of $F$ such that, with $Q:F\rightarrow F/Y$
the quotient map, $\|QT\|_{cb} <\epsilon$.

For $\mu\in\mc A^*$ say that $\mu$ is \emph{completely almost periodic},
denoted by $\mu\in\CAP(\mf A)$, if both $L_\mu$ and $R_\mu$ are completely
compact.
\end{definition}

Consider the case when $(M,\Delta)$ is a Hopf von-Neumann algebra and
$\mf A=M_*$ with the canonical operator space structure.  Then, if $M$ is
an injective von Neumann algebra, \cite[Theorem~2.4]{runde} shows that
$\CAP(\mf A)$ is a C$^*$-subalgebra of $M$.  Indeed, the proof shows that
\[ \CAP(\mf A) = \{ x\in M : \Delta(x) \in M \otimes M \}, \]
where here $\otimes$ denotes the minimal C$^*$-algebra tensor product,
namely the norm closure of $M\odot M$ in $M\vnten M$.  In particular, this
applies to $\mf A=A(G)$ when $G$ is an amenable or connected locally compact
group, \cite[Corollary~2.5]{runde}.

Hence for ``nice'' $\G$, Runde's algebra $\CAP(L^1(\G))$ agrees with those
$x\in L^\infty(\G)$ such that
\[ L_x : L^1(\G)\rightarrow L^\infty(\G); \quad \omega \mapsto
x \star \omega = (\omega\otimes\iota)\Delta(x) \]
can be cb-norm approximated by finite-rank maps.  Equivalently, this means
that $\Delta(x)$ can be norm approximated by finite-rank tensors in
$M\vnten M$.

\subsection{Counter-examples}\label{sec:e2}

The counter-example considered by Chou is as follows: for a compact group $G$,
let $E$ be the rank-one orthogonal projection onto the constant functions in
$L^2(G)$.  Then $E\in VN(G)$ and it is possible to analyse closely the orbit
map $A(G)\rightarrow VN(G); \omega\mapsto\omega\cdot E$, in particular,
$E\in AP(\hat G)$ if and only if $G$ is tall.  However, a careful
calculation shows that Chou's argument \cite[Proposition~3.1]{chou} (compare
also \cite{dr2}) does extend to $\CAP(\hat G)$.  We plan to explore in future
work if we can characterise when $E\in\CAP(\hat G)$; compare also with
Theorem~\ref{thm:cap_partial_case} below.

Instead, we now turn our attention to the fully quantum case, and explore
a remarkable result of Woronowicz in \cite{woro4}.  The quantum $E(2)$ group
was defined in \cite{woro5}, see also \cite{ps,vw}.

\begin{theorem}
Let $q\in (0,1)$ and let $\G$ be the quantum $E(2)$ group with parameter $q$.
Then $\CAP(\G)$ strictly contains $\mathbb{AP}(C_0(\G))\cong C(\G^\sap)$.
\end{theorem}
\begin{proof}
We use two results of Woronowicz.  Firstly, \cite{woro4} (see especially
Section~4) shows that for all $x\in C_0(\G)$, we have that $\Delta(x)
\in C_0(\G) \otimes C_0(\G)$; notice the lack of a multiplier algebra!
So immediately we see that $C_0(\G) \subseteq \CAP(\G)$.

Secondly, we use the classification of unitary corepresentations of
$C_0(\G)$ given in \cite[Theorem~2.1]{woro5}.  In finite-dimensions, the
classification is very restrictive.  Indeed, let $U\in M(C_0(\G))\otimes
\mathbb M_n$ be a finite-dimensional unitary corepresentation.  Then
there exist matrices $N,b\in\mathbb M_n$ with $N$ self-adjoint, $b$ normal,
and $N,|b|$ commuting.  Furthermore, if $b$ has
polar decomposition $b=u|b|$, then on $(\ker b)^\perp$, we have that
$u^*Nu=N+2$.  Clearly this cannot hold for \emph{bounded} operators, so
in this finite-dimensional setting, $(\ker b)^\perp=\{0\}$ so $b=0$.
Then \cite[Theorem~2.1]{woro5} further gives an expression for $U$ in
terms of $b$ and $N$.  However, as $b=0$, it follows that actually $U \in
\mathbb M_n(M(C_0(\G)))$ diagonalises, with diagonal entries powers of
$v\in M(C_0(\G))$.  Here $v$ is a unitary, one of the operators which
``generates'' $C_0(\G)$, compare \cite[Theorem~1.1]{woro5}.  We know
that $\Delta(v)=v\otimes v$ by \cite[Theorem~1.2]{woro5}, and so $v$ is
a one-dimensional, admissible, corepresentation.

It follows from this discussion that $\mc{AP}(\G)$ is spanned by
$\{ v^k : k\in\mathbb Z \}$ and so $\G^\sap$ is isomorphic to the circle
group (and hence is a classical compact group).  In particular,
$\mathbb{AP}(C_0(\G))$ is (much) smaller than $\CAP(\G)$.
\end{proof}

Let us draw one further conclusion from this.  Let $A\subseteq M(C_0(\G))$
be the maximal compact quantum semigroup; so $A$ is the maximal unital
C$^*$-subalgebra with $\Delta(A)\subseteq A\otimes A$.  That $A$ exists follows
from a free-product argument, compare \cite{wang}: if $B,C\subseteq
M(C_0(\G))$ are two such unital algebras, then the image of the free-product
$B*C$ in $M(C_0(\G))$ will be a unital C$^*$-bialgebra containing both $B$ and
$C$.  If $\G=G$ is a classical (semi)group then $A$ will be commutative, with
character space $K$ say, and $\Delta$ restricted to $A$ will induce a
continuous semigroup structure on $K$.  It follows that $K$ is actually the
``almost periodic'' compactification of $G$, compare \cite[Chapter~4.1]{berg}.
If $G$ were a group, then this agrees with the Bohr compactification.
However, if now $\G$ is again the quantum $E(2)$ group, then this maximal
$A$ will certainly contain $C_0(\G)\oplus\mathbb C 1$; in particular, again
$A$ will be (much) larger than $\mathbb{AP}(C_0(\G))$.

\subsection{Stronger notions}

Thus it appears that, in complete generality, even the notion of
``completely almost periodic'' is not strong enough to single out
$\mathbb{AP}(C_0(\G))$.  We shall make a stronger definition; the link with
Runde's definition is clarified in Proposition~\ref{prop:one} below.

\begin{definition}
Say that $x\in L^\infty(\G)$ is \emph{periodic} if $\Delta(x)$ is
a finite-rank tensor in $L^\infty(\G) \vnten L^\infty(\G)$.  Denote the
collection of periodic elements of $L^\infty(\G)$ by $\mc{P}^\infty(\G)$,
and denote the norm closure, in $L^\infty(\G)$, by $\mathbb{P}^\infty(\G)$.
\end{definition}

Note that similar definitions would hold for $M(C_0(\G))$ (or indeed for
$M(C_0^u(\G))$) and at this level of generality, it is not clear if they would
yield the same objects; thus we choose to emphasis this by using the notation
$\mc{P}^\infty$ not $\mc P$, and so forth.  Compare with Remark~\ref{rem:one}
and the comment after Theorem~\ref{thm:ss_okay_cmpt}, which suggest that
working with $M(C_0(\G))$ should probably give the same result, but that
$M(C_0^u(\G))$ might not.

\begin{remark}
In the purely algebraic setting, working with multiplier Hopf algebras,
So{\l}tan studied a very similar ideas for discrete quantum groups in
\cite{soltan1} (see also \cite[Proposition~4.6]{soltan}).
\end{remark}

\begin{lemma}
Let $\G$ be a locally compact quantum group.  Then
\[ \mc{P}^\infty(\G)
= \{ x\in L^\infty(\G) : L_x\text{ is finite-rank}\}
= \{ x\in L^\infty(\G) : R_x\text{ is finite-rank}\}. \]
Furthermore, $\mathbb{P}^\infty(\G)$ is a C$^*$-subalgebra of $L^\infty(\G)$,
and an $L^1(\G)$-submodule of $L^\infty(\G)$.
\end{lemma}
\begin{proof}
It follows immediately from the definition that $\mathbb{P}^\infty(\G)$ is a
C$^*$-algebra.  As $L_x(\omega) = x\star\omega =
(\omega\otimes\iota)\Delta(x)$, it is easy to see that if $\Delta(x)$ is a
finite-rank tensor, then $L_x$ is a finite-rank map.  Conversely, if
$L_x$ is finite-rank, then let $\{ x_i \}$ be a basis for the image of
$L_x$.  For each $\omega\in L^1(\G)$, there hence exist unique scalars
$\{a_i\}$ with $x\star\omega = \sum_i a_i x_i$.  Then the map $\omega\mapsto
a_i$ is bounded and linear, so there are $y_i$ in $L^\infty(\G)$ with
$x\star\omega = \sum_i \ip{y_i}{\omega} x_i$.  Equivalently,
$\Delta(x) = \sum_i y_i \otimes x_i$ is a finite-rank tensor.  An analogous
argument holds for $R_x$; or use again that $R_x = L_x^*\circ\kappa$ and
$L_x = R_x^*\circ\kappa$ (where $\kappa$ is the inclusion
$L^1(\G)\rightarrow L^1(\G)^{**}$).

To show that $\mathbb{P}^\infty(\G)$ is an $L^1(\G)$-submodule, it suffices
to show that $\mc{P}^\infty(\G)$ is a submodule.  For $\omega\in L^1(\G)$,
let $T_\omega: L^\infty(\G)\rightarrow L^\infty(\G)$ be the map
$x\mapsto \omega\star x$.  Then $L_{\omega\star x} = T_\omega \circ L_x$,
so $x\in\mc{P}^\infty(\G) \implies \omega\star x\in\mc{P}^\infty(\G)$.
Similarly, let $S_\omega:L^1(\G)\rightarrow L^1(\G)$ be the map $\tau
\mapsto \omega\tau$.  Then $L_{x\star\omega} = L_x \circ S_\omega$, and it
follows that $\mc{P}^\infty(\G)$ is also a right $L^1(\G)$-submodule.
\end{proof}

The following ultimately provides an alternative description of
$\mc{P}^\infty(\G)$.

\begin{theorem}
Let $T:L^1(\G)\rightarrow L^\infty(\G)$ be a completely bounded
right $L^1(\G)$-module homomorphism.  Then there exists $x\in L^\infty(\G)$
with $T = R_x$.
\end{theorem}
\begin{proof}
That $T$ is completely bounded again means that there is $y\in L^\infty(\G)
\vnten L^\infty(\G)$ with $T(\omega) = (\omega\otimes\iota)(y)$ for all
$\omega\in L^1(\G)$.  If we have that $y=\Delta(x)$, then $T(\omega) =
(\omega\otimes\iota)\Delta(x) = \omega\star x = R_x(\omega)$ as required.
Conversely, that $T$ is a right $L^1(\G)$-module homomorphism is equivalent to
\[ \ip{y}{\omega_1\omega_2\otimes\omega_3}
= \ip{T(\omega_1\omega_2)}{\omega_3}
= \ip{T(\omega_1)}{\omega_2\omega_3}
= \ip{y}{\omega_1\otimes\omega_2\omega_3}, \]
that is, $(\Delta\otimes\iota)(y) = (\iota\otimes\Delta)(y)$.  So we wish
to prove that this relation on $y$ forces that $y=\Delta(x)$ for some $x$.

If $L^1(\G)$ has a bounded approximate identity, this is a Banach algebraic
exercise.  For $\G$ cocommutative, we gave a proof in \cite[Theorem~6.5]{daws},
and an early preprint by the author\footnote{See Theorem~2.2 of
\texttt{arXiv:1107.5244v3 [math.OA]}}
shows that this is true for general $\G$ (with a relatively elementary proof).
However, actually a somewhat more general statement already exists in
the literature.  A \emph{(left) action} of $\G$ on a von Neumann algebra $N$
is an injective normal unital $*$-homomorphism $\alpha:N\rightarrow
L^\infty(\G)\vnten N$ with $(\iota\otimes\alpha)\alpha = (\Delta\otimes\iota)
\alpha$.  For actions of Kac algebras, it was shown in
\cite[Th\'eor\`eme~IV.2]{es} (see also \cite[D\'efinition~II.8]{es}) that
\[ \alpha(N) = \{ z\in L^\infty(\G)\vnten N : (\iota\otimes\alpha)(z)
= (\Delta\otimes\iota)(z) \}. \]
In \cite[Section~2]{vaes} the necessary preliminary steps to prove
this result for locally compact quantum groups are given, although a full
proof is not shown.  As $\Delta$ is an action of $\G$ on $L^\infty(\G)$,
our claim immediately follows from this general theory.
\end{proof}

\begin{proposition}\label{prop:one}
Let $\G$ be a locally compact quantum group.  Then
$\mathbb{AP}(C_0(\G)) \subseteq \mathbb{P}^\infty(\G) \subseteq \CAP(L^1(\G))$.
Indeed, $x\in\mathbb{P}^\infty(\G)$ if and only if $L_x$ can be cb-norm
approximated by finite-rank module maps $L^1(\G)\rightarrow L^\infty(\G)$.
\end{proposition}
\begin{proof}
It is clear that a matrix element of an admissible corepresentation is
periodic (as the corepresentation is finite-dimensional) and so
$\mathbb{AP}(C_0(\G)) \subseteq \mathbb{P}^\infty(\G)$.  If $x$ is periodic,
then as $\Delta$ is an isometry, it follows that $\Delta(x)$ can be norm
approximated by elements of form $\Delta(y)$ with $\Delta(y)$ a finite-rank
tensor.  It is immediate that $x\in \CAP(L^1(\G))$, and that $L_x$ can be
cb-norm approximated by finite-rank module maps.  The converse now follows
from the previous theorem, as every finite-rank module map is of the form
$L_y$ for some $y$ with $\Delta(y)$ a finite-rank tensor.
\end{proof}

Thus the collection $\mathbb{P}^\infty$, which can be defined purely in terms
of the Banach algebra $L^1(\G)$, is a weakening of $\mathbb{AP}$ and a
strengthening of $\CAP$.  In the next section we shall provide cases
(in particularly, when $\G$ is a Kac algebra) when
$\mathbb{AP} = \mathbb{P}^\infty$.  In the following section we embark on a
programme to determine $\CAP(A(G))$ for various classes of groups $G$.

\subsection{When periodic implies almost periodic}\label{sec:pap}

Our aim here is to show that if $x\in\mathbb{P}^\infty(\G)$, then under further
assumptions on $x$, also $x\in\mathbb{AP}(C_0(\G))$.  Firstly, we need to
decide upon reasonable ``further assumptions''.
Proposition~\ref{prop:admiss_iff_dsinv} immediately implies the following.

\begin{lemma}\label{lem:in_ap_in_ss}
Let $x\in\mc{AP}(\G)$.  If we consider $\mc{AP}(\G)$ as a subalgebra
of $M(C_0(\G))$, then $x\in D(S) \cap D(S^{-1})$.  Similarly, if we consider
$\mc{AP}(\G) \subseteq M(C_0^u(\G))$, then $\mc{AP}(\G) \subseteq
D(S_u) \cap D(S_u^{-1})$.
\end{lemma}

Our aim will be to show that in fact $\mc{AP}(\G) = \mathbb{P}^\infty(\G)\cap
D(S) \cap D(S^{-1})$.  Notice that if $S$ is actually bounded (if $\G$ is
a Kac algebra) then immediately we have that $\mc{AP}(\G) = \mathbb{P}(\G)$.

We start with some results about the antipode.  As $S$ is unbounded, the
natural candidate for an involution on $L^1(\G)$ is unbounded.  Instead,
following for example \cite[Section~4]{kus1}, we define $L^1_\sharp(\G)$
to be the collection of $\omega\in L^1(\G)$ such that there is $\omega^\sharp
\in L^1(\G)$ satisfying $\ip{x}{\omega^\sharp} = \overline{\ip{S(x)^*}{\omega}}$
for all $x\in D(S)$.  Then $\omega\mapsto \omega^\sharp$ defines an
involution on $L^1_\sharp(\G)$.  For $\omega\in L^1(\G)$ define $\omega^*$
by $\ip{x}{\omega^*} = \overline{\ip{x^*}{\omega}}$ for $x\in L^\infty(\G)$.

\begin{lemma}\label{lem:slices_dom_s}
Let $x\in D(S)\subseteq L^\infty(\G)$ and let $\omega\in L^1_\sharp(\G)$.
Then $x \star \omega \in D(S)$ with $S(x\star\omega)
= \omega^{\sharp *} \star S(x)$, and $\omega\star x\in D(S)$ with
$S(\omega\star x) = S(x) \star \omega^{\sharp *}$.

Let $y\in D(S^{-1})$ and let $\tau\in L^1_\sharp(\G)^*$.  Then $y\star\tau,
\tau\star y\in D(S^{-1})$ with $S^{-1}(y\star\tau)=\tau^{*\sharp}
\star S^{-1}(y)$ and $S^{-1}(\tau\star y) = S^{-1}(y)\star\tau^{*\sharp}$.
\end{lemma}
\begin{proof}
By \cite[Appendix~A]{bds}, for example, to show that $x\star\omega\in D(S)$
with $S(x\star\omega) = \omega^{\sharp *} \star S(x)$, it suffices to show
that for all $\tau\in L^1_\sharp(\G)$, we have that
\[ \ip{x\star\omega}{\tau^\sharp}
= \ip{\omega^{\sharp *} \star S(x)}{\tau^*}. \]
However, this follows from a simple calculation:
\begin{align*} \ip{x\star\omega}{\tau^\sharp} &=
\ip{x}{\omega \star \tau^\sharp}
= \ip{x}{(\tau\star \omega^\sharp)^\sharp}
= \overline{ \ip{S(x)^*}{\tau\star \omega^\sharp} }
= \overline{ \ip{\omega^\sharp \star S(x)^*}{\tau} }
= \ip{\omega^{\sharp *} \star S(x)}{\tau^*}.
\end{align*}
The second claim is entirely analogous.

For the second part, we use that $S^{-1} = *\circ S\circ *$, and that
$\Delta$ is a $*$-homomorphism.  So $y^*\in D(S)$ and
$\tau^*\in L^1_\sharp(\G)$, and thus $y^*\star\tau^*\in D(S)$ with
$S(y^*\star\tau^*) = \tau^{*\sharp *} \star S(y^*)$.  Hence
$y\star\tau \in D(S^{-1})$ and
$S^{-1}(y\star\tau) = S(y^* \star \tau^*)^* =
\tau^{*\sharp}\star S^{-1}(y)$.  The other claim follows similarly.
\end{proof}

The following is essentially a restatement of \cite[Proposition~4.6]{bds};
but here our sums are finite, and so we can ignore convergence issues.

\begin{proposition}\label{prop:multsresult}
Let $x\in L^\infty(\G)$ be such that $\Delta(x) = \sum_{i=1}^n
a_i \otimes b_i$ where for each $i$, we have that $b_i^*\in D(S)$.
Then the map $L:L^\infty(\hat\G)\rightarrow \mc B(L^2(\G));
\hat x \mapsto \sum_{i=1}^n S(b_i^*)^* \hat x a_i$ maps into
$L^\infty(\hat\G)$ and is the adjoint of a (completely bounded) left
multiplier of $L^1(\hat\G)$.  In particular, $(\iota\otimes L)(W^*) =
(x\otimes 1)W^*$, equivalently, $\sum_i (1\otimes S(b_i^*)^*) \Delta(a_i)
= x\otimes 1$.
\end{proposition}

\begin{corollary}\label{corr:multsresult}
Let $x\in L^\infty(\G)$ be such that $\Delta(x) = \sum_{i=1}^n
a_i \otimes b_i$ where for each $i$, we have that $a_i\in D(S)$.
Then $1\otimes x = \sum_i (S(a_i)\otimes 1)\Delta(b_i)$.
\end{corollary}
\begin{proof}
We follow \cite[Section~4]{kvvn}, and define the \emph{opposite quantum group}
to $\G$ to be $\G^\op$, where $L^\infty(\G^\op) = L^\infty(\G)$ and
$\Delta^\op = \sigma\Delta$.  Then we find that $R^\op = R$ and
$(\tau^\op_t) = (\tau_{-t})$.  It follows that $S^\op = R^\op \tau^\op_{-i/2}
= R \tau_{i/2} = S^{-1} = * \circ S \circ *$.

So $\Delta^\op(x) = \sum_i b_i \otimes a_i$ and for each $i$, we have that
$a_i^* \in D(S^\op)$.  So the previous proposition, now applied to $\G^\op$,
shows that $x\otimes 1 = \sum_i (1\otimes S^\op(a_i^*)^*) \Delta^\op(b_i)$,
or equivalently, $1\otimes x = \sum_i (S(a_i)\otimes 1)\Delta(b_i)$.
\end{proof}

We can now prove the main theorem of this section; the idea of this
construction is well-known in Hopf algebra theory (but of course we have
to work harder in the analytic setting).

\begin{theorem}\label{thm:pss_ap}
Let $x\in\mc{P}^\infty(\G) \cap D(S)\cap D(S^{-1})$.  There exists an
admissible corepresentation $U=(U_{ij})$ of $C_0(\G)$ such that $x$ is a matrix
element of $U$.  In particular, $x\in M(C_0(\G))$ and $x\in\mc{AP}(C_0(\G))$.
\end{theorem}
\begin{proof}
That $x\in\mc{P}^\infty(\G)$ is equivalent to $R_x:L^1(\G)\rightarrow
L^\infty(\G); \omega\mapsto \omega\star x$ having a finite-dimensional
image, say $X$.  Then $R_x(L^1_\sharp(\G))$ is dense in $X$, as
$L^1_\sharp(\G)$ is dense in $L^1(\G)$.  As $X$ is finite-dimensional,
$R_x(L^1_\sharp(\G)) = X$, and so we can find $(\omega_i)_{i=1}^n \subseteq
L^1_\sharp(\G)$ with $\{ \omega_i \star x : 1\leq i\leq n\}$ forming a basis
for $X$.  Set $x_i = \omega_i\star x$.

As $\{x_i\}$ is a linearly independent set, the map $L^1(\G) \rightarrow
\mathbb C^n; \omega\mapsto (\ip{x_i}{\omega})_{i=1}^n$ is a linear surjection.
Hence the set $\{ (\ip{x_i}{\omega})_{i=1}^n : \omega\in L^1_\sharp(\G) \}$
is a dense linear subspace of $\mathbb C^n$, and hence equals $\mathbb C^n$.
So we can find $(\tau_i)_{i=1}^n \subseteq L^1_\sharp(\G)$ with
$\ip{x_j}{\tau_i} = \delta_{i,j}$ for $1\leq i,j\leq n$.  Let $y_i=
x \star \tau_i$.

For $\omega\in L^1(\G)$ there are unique $(a_i)_{i=1}^n\subseteq\mathbb C$
with $\omega\star x = \sum_i a_i x_i$.  Then $a_i = \ip{\omega\star x}{\tau_i}
= \ip{x\star\tau_i}{\omega} = \ip{y_i}{\omega}$, and so for $\tau\in L^1(\G)$,
\[ \ip{\Delta(x)}{\tau\otimes\omega} = \sum_i a_i \ip{x_i}{\tau}
= \sum_i \ip{x_i\otimes y_i}{\tau\otimes\omega}. \]
It follows that $\Delta(x) = \sum_i x_i \otimes y_i$.  Notice that then
$x_i = \omega_i \star x = \sum_j \ip{y_j}{\omega_i} x_j$, so as
$\{x_i\}$ is a linearly independent set, we conclude that
$\ip{y_j}{\omega_i} = \delta_{i,j}$ for all $i,j$.  In particular,
$\{ y_j \}$ is also a linearly independent set.

Set $U_{ij} = x_j \star \tau_i = \omega_j \star x \star \tau_i$.
Then
\[ \Delta^2(x) = \sum_j \Delta(x_j)\otimes y_j
= \sum_i x_i\otimes\Delta(y_i). \]
It follows that $\Delta(y_i) = \sum_j U_{ij} \otimes y_j$ for all $i$.
Then
\[ \Delta^2(y_i) = \sum_j \Delta(U_{ij}) \otimes y_j
= \sum_k U_{ik} \otimes \Delta(y_k)
= \sum_{k,j} U_{ik} \otimes U_{kj} \otimes y_j. \]
As $\{ y_j \}$ is also a linearly independent set, this shows that
$U=(U_{ij})$ is a corepresentation.

Using Lemma~\ref{lem:slices_dom_s}, notice that $y_i, x_i\in D(S)$
for all $i$, and that $U_{ij} \in D(S)$ for all $i,j$.  As $\Delta(x)=
\sum_i x_i\otimes y_i$, Corollary~\ref{corr:multsresult} shows that
\[ 1\otimes x = \sum_i (S(x_i)\otimes 1)\Delta(y_i)
= \sum_{i,j} S(x_i) U_{ij} \otimes y_j. \]
Hence $x \in \lin\{ y_i \}$.  We now notice that for each $i$,
$\Delta(y_i^*) = \sum_j U_{ij}^* \otimes y_j^*$, and so
Proposition~\ref{prop:multsresult} shows that
\[ y_i^*\otimes 1 = \sum_j (1\otimes S(y_j)^*)\Delta(U_{ij}^*)
= \sum_{jk} U_{ik}^* \otimes S(y_j)^*U_{kj}^*. \]
Hence $y_i \in \lin\{ U_{ik} : 1\leq k\leq n \}$ for each $i$, and
we conclude that $x$ is a matrix element of $U$.

The preceding argument is partly inspired by \cite[Section~4]{bds};
in particular, \cite[Theorem~4.7]{bds} shows that if the representation
$L^1(\G) \rightarrow \mathbb M_n; \omega \mapsto (\ip{U_{ij}}{\omega})$
is non-degenerate, then $U$ is invertible with inverse $(S(U_{ij}))_{i,j=1}^n$.
We claim that this representation is indeed non-degenerate.  For
$\xi\in\mathbb C^n$, as $\{y_i\}$ is a linearly independent set, there
is $\tau\in L^1(\G)$ with $\xi = ( \ip{y_j}{\tau} )_{j=1}^n$.  Then
\[ (\ip{U_{ij}}{\omega}) \xi
= \Big( \sum_j \ip{U_{ij}}{\omega} \ip{y_j}{\tau} \Big)_{i=1}^n
= \big( \ip{y_i}{\omega\star\tau} \big)_{i=1}^n. \]
We need to show that as $\omega,\tau$ vary, we get a subset whose linear
span is all of $\mathbb C^n$.  However, this is clear, because
$\lin\{ \omega\star \tau : \omega,\tau\in L^1(\G) \}$ is a dense subspace
of $L^1(\G)$, and the set $\{y_i\}$ is linearly independent.

We hence conclude that $U$ is invertible.  It is known that invertible
representations are members of $M(C_0(\G)\otimes\mathbb M_n)
= M(C_0(\G)) \otimes \mathbb M_n$, see \cite[Section~4]{woro},
\cite[Page~441]{bs2} and \cite[Theorem~4.9]{bds}.  So $x\in M(C_0(\G))$.

It remains to show that $U$ is admissible, that is, $U^\top$ is invertible,
or equivalently, that $\overline{U} = (U_{ij}^*)_{i,j=1}^n$ is invertible.
By hypothesis, also $x\in D(S^{-1})$.  Arguing as before, we can find
$(\omega_i')_{i=1}^n \subseteq L^1_\sharp(\G)^*$ with $x_i = \omega_i'\star x$.
Similarly, we find $(\tau_i')_{i=1}^n \subseteq L^1_\sharp(\G)^*$ with
$\ip{x_j}{\tau_i'} = \delta_{ij}$.  Setting $y_i' = x\star\tau_i'$, we follow
the argument above to conclude that $\Delta(x) = \sum_i x_i\otimes y_i'$.
As $\{ x_i\}$ is a linearly independent set, actually $y_i'=y_i$ for all $i$.
Then $U'_{ij} = x_j \star \tau_i'$ is a corepresentation, but now
$U'_{ij} \in D(S^{-1})$ for all $i,j$.  However,
\[ \sum_j U_{ij} \otimes y_j = \Delta(y_i) = \Delta(y_i')
= \sum_j U'_{ij} \otimes y_j' = \sum_j U'_{ij} \otimes y_j, \]
so using that $\{y_j\}$ is a linearly independent set, it follows that
$U_{ij} = U'_{ij}$ for all $i,j$.  Thus $U_{ij}^*\in D(S)$ for all $i,j$,
and so \cite[Theorem~4.7]{bds} shows that $\overline{U}$ is invertible
(as obviously $L^1(\G)\rightarrow\mathbb M_n;
\omega\mapsto (\ip{U_{ij}^*}{\omega})_{i,j=1}^n = 
(\overline{\ip{U_{ij}}{\omega^*}})_{i,j=1}^n$ is non-degenerate).
\end{proof}

\begin{remark}\label{rem:one}
An immediate corollary of the proof is that if $x\in\mc{P}^\infty(\G)
\cap D(S)$, then $x\in M(C_0(\G))$ and $x$ is a matrix element of an
invertible, finite-dimensional corepresentation of $C_0(\G)$.
\end{remark}

Combining Lemma~\ref{lem:in_ap_in_ss} and Theorem~\ref{thm:pss_ap}
immediately gives the following.

\begin{corollary}
We have that $\mc{AP}(C_0(\G)) = \mc{P}^\infty(\G) \cap D(S)
\cap D(S^{-1})$.
\end{corollary}

For Kac algebras, this takes on a very pleasing form, because $S=R$ is
bounded.  In particular, this answers a question ask by So\l tan,
see \cite[Page~1260]{soltan}, as to whether a finite-dimensional, unitary
corepresentation of a discrete Kac algebra is automatically admissible--
the answer is ``yes'', as would be true for any Kac algebra, and any
finite-dimensional corepresentation.

\begin{corollary}\label{corr:kacokay}
Let $\G$ be a Kac algebra.  Then $\mc{AP}(C_0(\G)) = \mc{P}^\infty(\G)$,
and $\mathbb{AP}(C_0(\G)) = \mathbb{P}^\infty(\G)$.
\end{corollary}

So at least for a Kac algebra, this corollary provides a way, just starting
from $L^1(\G)$, of finding $\mc{AP}(C_0(\G))$ and $\mathbb{AP}(C_0(\G))$.
In the general case, $D(S)\cap D(S^{-1})$ will be smaller than $L^\infty(\G)$,
and so in principle it might be hard to calculate $\mc{P}^\infty(\G)
\cap D(S)\cap D(S^{-1})$.  Compare with Theorem~\ref{thm:ss_okay_cmpt}
and Conjecture~\ref{conj:one} below.

\section{The cocommutative case}\label{sec:cocomcase}

So\l tan showed in \cite[Section~4.2]{soltan} that if $G$ is a locally
compact group and $\G = \hat G$, then $\mathbb{AP}(C_0^u(\G)) =
\mathbb{AP}(C^*(G))$ is the closed linear span of the translation operators
$\{ L_s : s\in G \}$ inside $M(C^*(G))$.  Actually, much the same proof work
for $C^*_r(G) = C_0(\G)$.  So in light of Proposition~\ref{prop:same_red_uni},
we see that if $G_d$ denotes $G$ with the discrete topology, then
the compactification of $\G$ in ${\sf LCQG}$ is $\widehat{G_d}$

Let us give a different, short proof of this, using our previous results.
In the literature (see \cite{chou} for example) it is common to write
$C^*_\delta(G)$ for the C$^*$-subalgebra of $VN(G)=L^\infty(\G)$ generated
by the left-translation maps $\{ \lambda(s) : s\in G \}$.  Let us write
$\mathbb C[G]$ for the algebra (not norm closed) generated by the operators
$\lambda(s)$.  We first note that each $\lambda(s)\in M(C_0(\G))$, and that
$\Delta$ restricts to $\mathbb C[G]$ turning $(\mathbb C[G], \Delta)$ into a
unital bialgebra.  Thus $(C^*_\delta(G),\Delta)$ is a unital C$^*$-bialgebra.
It is easy to verify the density conditions to show that
$(C^*_\delta(G),\Delta)$ is a compact quantum group.
Hence $C^*_\delta(G) \subseteq \mathbb{AP}(C_0(\G))$.
However, Chou showed in \cite[Proposition~2.3]{chou} (translated to our
terminology) that $\mc P^\infty(\G) = \mathbb C[G]$, from which it follows
that $\mathbb P^\infty(\G) = C^*_\delta(G)$.  So, with reference to
Proposition~\ref{prop:one}, it follows that $\mathbb{AP}(C_0(\G)) =
C^*_\delta(G)$.

\subsection{Completely almost periodic elements}

We now wish to investigate Runde's definition $\CAP(A(G))$.  Firstly,
we have the following result, which holds for any compact Kac algebra.

\begin{theorem}\label{thm:cap_cmpt_kac_case}
Let $\G$ be a compact Kac algebra.  If $x\in L^\infty(\G)$ with
$\Delta(x) \in L^\infty(\G) \otimes L^\infty(\G)$ (the C$^*$-algebraic
minimal tensor product) then $x\in C(\G)$.
\end{theorem}
\begin{proof}
As $\G$ is Kac, we have that $R=S=S^{-1}$ and the Haar state $\varphi$ is
a trace.  Let $\xi_0\in L^2(\G)$ be the cyclic vector for $\varphi$.
Then, for $y,z\in L^\infty(\G)$ and $a,b\in C(\G)$,
\begin{align*} & \big( (\varphi\otimes\iota)(W(y\otimes z)W^*) a \xi_0
  \big| b \xi_0 \big)
= \big( (y\otimes z) W^*(\xi_0\otimes a\xi_0) \big|
  W^*(\xi_0\otimes b\xi_0) \big) \\
&= \big( \Delta(b^*)(y\otimes z) \Delta(a) \xi_0\otimes \xi_0 \big|
  \xi_0\otimes\xi_0 \big)
= (\varphi\otimes\varphi)\big( \Delta(b^*)(y\otimes z) \Delta(a) \big) \\
&= (\varphi\otimes\varphi)\big( \Delta(ab^*)(y\otimes z) \big).
\end{align*}
using that $\varphi$ is a trace.  By \cite[Theorem~2.6(4)]{woro2} we 
know that
\[ R\big( (\varphi\otimes\iota)\big( (ab^*\otimes 1)\Delta(y) \big) \big)
= (\varphi\otimes\iota)\big( \Delta(ab^*)(y\otimes 1) \big). \]
Hence we get
\begin{align*}
\varphi\big( R\big( (\varphi\otimes\iota)\big( (ab^*\otimes 1)\Delta(y)
\big) \big)   z \big)
= \varphi\big( R(z) \big( (\varphi\otimes\iota)\big( (ab^*\otimes 1)\Delta(y)
\big) \big) \big),
\end{align*}
using that $\varphi \circ R = \varphi$.  This is then equal to
\begin{align*}
& (\varphi\otimes\varphi)\big( (ab^*\otimes R(z))\Delta(y) \big)
= \big( W^*(\xi_0\otimes y\xi_0) \big| ba^*\xi_0 \otimes R(z^*)\xi_0 \big) \\
&= \big( (\iota\otimes\hat\omega)(W^*) \xi_0 \big| ba^*\xi_0 \big)
= \varphi\big( ab^* c \big) = \varphi\big(b^*ca\big)
= \big( c a\xi_0 \big| b\xi_0 \big),
\end{align*}
where $\hat\omega = \omega_{y\xi_0,R(z^*)\xi_0} \in L^1(\hat\G)$ and
$c = (\iota\otimes\hat\omega)(W^*)\in C(\G)$.  As $a,b\in C(\G)$ were
arbitrary, this shows that
\[ (\varphi\otimes\iota)(W(y\otimes z)W^*) \in C(\G). \]
We remark that a similar idea to this construction, in a very different
context, appears in \cite[Section~3]{rx}.

For $\epsilon>0$ we can find $\tau \in L^\infty(\G)\odot L^\infty(\G)$
with $\|\Delta(x)-\tau\|<\epsilon$.  Thus
\[ \| (\varphi\otimes\iota)(W\Delta(x)W^*) -
(\varphi\otimes\iota)(W\tau W^*) \| < \epsilon. \]
However, as $W^*\Delta(x)W=1\otimes x$, this shows that there is some
$d\in C(\G)$ with $\|x-d\|<\epsilon$.  So $x\in C(\G)$ as required.
\end{proof}

\begin{corollary}
Let $\G$ be a compact Kac algebra with $L^\infty(\G)$ injective
(for example, $\G=\hat G$ for a discrete amenable group $G$).  Then
$\CAP(L^1(\G)) = C(\G)$.
\end{corollary}

To deal with the non-injective case would presumably require a much more
detailed study of the notion of a completely compact map.  We do not currently
see how to adapt our ideas to the non-Kac setting.

In the rest of this section, we start a programme of extending
Theorem~\ref{thm:cap_cmpt_kac_case}.  Having a trace seemed very important to
this proof, so to make progress in the non-compact case we shall restrict
ourselves to the case when $\G=\hat G$ for a [SIN] group $G$.  We shall prove
the following result:

\begin{theorem}\label{thm:cap_partial_case}
Let $G$ be a [SIN] group, and let $x_0\in VN(G)$ be such that
$\Delta^2(x_0) \in VN(G)\otimes VN(G)\otimes VN(G)$.
Then $x_0\in C^*_\delta(G)$.
\end{theorem}

By definition, $G$ is [SIN] group if $G$ contains
arbitrarily small neighbourhoods of the identity which are invariant under
inner automorphisms.  Equivalently, by \cite[Proposition~4.1]{tay},
$VN(G)$ is a finite von Neumann algebra.  Discrete, compact and
abelian groups are all [SIN] groups.  A connected group $G$ is [SIN] if
and only if the quotient of $G$ by its centre is compact, if and only if
$G\rightarrow G^\sap$ is injective, see \cite[Chapter~12]{palmer}.
Also we note that [SIN] groups are always unimodular.

Recall that the fundamental unitary $W$ of $VN(G)$ satisfies
$W\xi(s,t) = \xi(ts,t)$ for $\xi\in L^2(G\times G), s,t\in G$.
As $A(G)$ is commutative, we of course have that $\omega\star x =
x\star\omega$ for $\omega\in A(G)$ and $x\in VN(G)$.  We write the module
action of $VN(G)$ on $A(G)$ by juxtaposition.

\begin{lemma}
Let $G$ be unimodular, and let $\xi\in L^2(G)$ satisfy $\xi(ab)=
\xi(ba)$ for (almost) every $a,b\in G$.  For all $z\in VN(G\times G)$
we have that $(\omega_\xi\otimes\iota)(WzW^*) \in VN(G)$.  Furthermore,
for $x,y\in VN(G)$ we have that $(\omega_\xi\otimes\iota)(W(x\otimes y)W^*)
= \varphi_x \star y$ where $\varphi_x = (\omega_\xi x)\circ R \in A(G)$.
\end{lemma}
\begin{proof}
Let $s,t\in G$ and let $w = (\omega_\xi\otimes\iota)(W(\lambda(s)\otimes
\lambda(t))W^*)$.  Let $f,g\in L^2(G)$, and calculate:
\begin{align*}
\big( w f \big| g \big)
&= \big( (\lambda(s) \otimes\lambda(t)) W^*(\xi\otimes f) \big|
  W^*(\xi\otimes g) \big) \\
&= \int_G \int_G W^*(\xi\otimes f)(s^{-1}a,t^{-1}b)
  \overline{ W^*(\xi\otimes g)(a,b) } \ da \ db  \\
&= \int_G \int_G \xi(b^{-1}ts^{-1}a) \overline{\xi(b^{-1}a)}
   \ da \ f(t^{-1}b) \overline{g(b)} \ db \\
&= \int_G \int_G \xi(ts^{-1}ab^{-1}) \overline{\xi(ab^{-1})}
   \ da \ f(t^{-1}b) \overline{g(b)} \ db
   \qquad\text{by assumption on $\xi$}\\
&= \int_G \int_G \xi(ts^{-1}a) \overline{\xi(a)}
   \ da \ f(t^{-1}b) \overline{g(b)} \ db
   \qquad\text{as $G$ is unimodular}\\
&= \ip{\lambda(st^{-1})}{\omega_\xi} \big( \lambda(t) f \big| g \big).
\end{align*}
Thus $w = \ip{\lambda(st^{-1})}{\omega_\xi} \lambda(t)\in VN(G)$.
Now, 
\[ \varphi_{\lambda(s)} \star \lambda(t)
= \ip{\lambda(t)}{\varphi_{\lambda(s)}} \lambda(t)
= \ip{\lambda(t^{-1})}{\omega_\xi \lambda(s)} \lambda(t)
= w, \]
and so the stated formula holds when $x=\lambda(s)$ and $y=\lambda(t)$.
By separate weak$^*$-continuity, the formula holds for all $x,y$.
Then, by weak$^*$-continuity again, it follows that for any $z\in
VN(G\times G)$ we do indeed have that
$(\omega_\xi\otimes\iota)(WzW^*) \in VN(G)$.
\end{proof}

The obvious use of this result is that $(\omega_\xi\otimes\iota)(W\Delta(x)W^*)
= x$ for any $x\in VN(G)$ and any suitable $\xi$.  Thus this is very similar
to the argument in Theorem~\ref{thm:cap_cmpt_kac_case} above.

Note that $\varphi_x = R(x)(\omega_\xi\circ R)$ and that if $\xi$ is
real-valued, then $\omega_\xi\circ R=\omega_\xi$, as for $x\in VN(G)$, we have
that $R(x) = \hat Jx^*\hat J$, and here $\hat J:L^2(G)\rightarrow L^2(G)$
is just the map of pointwise conjugation.  If $G$ is a [SIN] group,
then whenever $V$ is an invariant neighbourhood of the identity, then
$\xi_V = |V|^{-1/2} \chi_V$ will be an invariant unit vector in $L^2(G)$.
Arguing as in \cite[Section~3]{tay} we see that $\omega_\xi$ is a tracial
state in $A(G)$.  Let $\mc U$ be an ultrafilter refining the order filter
on the set $I$ of invariant neighbourhoods of the identity in $G$, and let
$\Phi$ be the weak$^*$-limit, taken in $VN(G)^*$, along $\mc U$,
of the net $(\omega_{\xi_V})$.  Thus $\Phi$ is a trace on $VN(G)$.  Define
\[ \alpha:VN(G\times G)\rightarrow VN(G); \quad
\alpha(x) = \lim_{\mc U} (\omega_{\xi_V}\otimes\iota)(W(x\otimes y)W^*), \]
where the limit is in the weak$^*$-topology on $VN(G)$.
For $\omega\in A(G)$,
\begin{align*}
\ip{\alpha(x\otimes y)}{\omega}
&= \lim_{\mc U} \ip{(R(x)\omega_{\xi_V})\star y}{\omega}
= \lim_{\mc U} \ip{\Delta(y)(R(x)\otimes 1)}{\omega_{\xi_V} \otimes \omega}
= \Phi\big( (\omega\star y) R(x) \big). \end{align*}
Let $(H,\pi,\xi_0)$ be the GNS construction for $\Phi$.  Hence
\[ \ip{\alpha(x\otimes y)}{\omega}
= \big( \pi(\omega\star y)\xi_0 \big| \pi(R(x^*)) \xi_0 \big). \]

We wish to find $H$ is a more concrete way, for which we turn to the
notion of an ultrapower of a Banach space, \cite{hein}.
Let $\ell^\infty(L^2(G),I)$ be the Banach space of bounded families of
vectors in $L^2(G)$ indexed by $I$.  Define a degenerate inner-product
on $\ell^\infty(L^2(G),I)$ by
\[ \big( (\xi_i) \big| (\eta_i) \big) = \lim_{i\rightarrow\mc U}
(\xi_i|\eta_i). \]
The null-space is $N_{\mc U} = \{ (\xi_i) : \lim_{i\rightarrow\mc U}
\|\xi_i\|=0 \}$ and $\ell^\infty(L^2(G),I) / N_{\mc U}$ becomes a Hilbert
space, denoted by $(L^2(G))_{\mc U}$.  The equivalence class defined by
$(\xi_i)$ will be denoted by $[\xi_i]$.  In particular, set $\xi_1
= [\xi_V]$.  Any $T\in\mc B(L^2(G))$ acts on $(L^2(G))_{\mc U}$ by
$T[\xi_i] = [T(\xi_i)]$.  It is now easy to verify that the map
\[ \pi(x)\xi_0 \mapsto [x(\xi_V)] = x\xi_1 \qquad (x\in VN(G)) \]
is an isometry, and so extends to an isometric embedding $H\rightarrow
(L^2(G))_{\mc U}$.  We shall henceforth identify $H$ with a closed subspace
of $(L^2(G))_{\mc U}$.

\begin{lemma}
For $x\in VN(G)$ and $\omega\in A(G)$, we have that $(\omega\star x)\xi_1
= (\omega\otimes\iota)(W^*) x\xi_1$.  Hence $H$ is an invariant subspace of
$(L^2(G))_{\mc U}$ for the action of $C_0(G)$.
\end{lemma}
\begin{proof}
A direct calculation easily establishes that
\[ \lim_{V\rightarrow\{e\}} \| W(f\otimes\xi_V) - f\otimes\xi_V \|
= 0 \qquad (f\in L^2(G)). \]
Let $g\in L^2(G)$, and let $(\overline{g}\otimes\iota):L^2(G\times G)
\rightarrow L^2(G)$ be the operator $\xi\otimes\eta \mapsto (\xi|g)\eta$.
Then, for $\omega=\omega_{f,g}\in A(G)$,
\begin{align*} \| & (\omega\star x)\xi_1 -
  (\omega\otimes\iota)(W^*) x\xi_1 \| \\
&= \lim_{V\rightarrow\mc U} \| (\overline{g}\otimes\iota)W^*(1\otimes x)W
  (f\otimes\xi_V)
  - (\overline{g}\otimes\iota)W^*(1\otimes x)(f\otimes\xi_V) \| \\
&= \lim_{V\rightarrow\mc U} \| (\overline{g}\otimes\iota)W^*(1\otimes x)
\big( W(f\otimes\xi_V) - f\otimes\xi_V \big) \| = 0,
\end{align*}
as required.  As $\{ (\omega\otimes\iota)(W^*) : \omega\in A(G) \}$
is dense in $C_0(G)$, by continuity, it follows that $C_0(G)$, acting on
$(L^2(G))_{\mc U}$, maps $H$ to $H$.
\end{proof}

Let $\hat\pi:C_0(G)\rightarrow\mc B(H)$ the resulting $*$-homomorphism.
For $\xi,\eta\in H$, it follows that $\omega_{\xi,\eta}\circ\hat\pi:
C_0(G)\rightarrow\mathbb C$ is a functional, and so defines a measure in
$M(G) = C_0(G)^*$.  Then left convolution by the measure defines a member
of $VN(G)$ (actually, of $M(C^*_r(G))$) which we shall denote by
$\mu_{\xi,\eta}$.  Notice that for $\mu\in M(G)$, the convolution operator
so defined is $(\iota\otimes\mu)(W^*)$, which makes sense, as
$W^*\in M(C^*_r(G)\otimes C_0(G))$.
For the following result, notice that for $x\in VN(G)$,
we have that $R(x^*)\xi_1 = [ \overline{x\xi_V} ]$ and so
it follows that $\mu_{y\xi_1, R(x)^*\xi_1} = \mu_{x\xi_1, R(y)^*\xi_1}$.

\begin{proposition}\label{prop:what_alpha_is}
We have that $\alpha(x\otimes y) = \mu_{y\xi_1, R(x)^*\xi_1}$.
Consequently, for $z\in VN(G)$ with $\Delta(z)\in VN(G)\otimes VN(G)$,
we have that $z$ is in the norm closure of $M(G)$ inside $VN(G)$.
\end{proposition}
\begin{proof}
For $\omega\in A(G)$, we have that
\begin{align*} \ip{\alpha(x\otimes y)}{\omega}
&= \big( \pi(\omega\star y)\xi_0 \big| \pi(R(x^*)) \xi_0 \big)
= \big( (\omega\otimes\iota)(W^*) y\xi_1 \big| R(x^*)\xi_1 \big) \\
&= \ip{ (\omega\otimes\iota)(W^*) }{\omega_{y\xi_1,R(x^*)\xi_1}\circ\hat\pi}
= \ip{ \mu_{y\xi_1,R(x^*)\xi_1} }{\omega}.
\end{align*}
For $\epsilon>0$, we can find $\tau\in VN(G)\odot VN(G)$ with
$\|\Delta(z)-\tau\|<\epsilon$.  Then
\[ \| z - \alpha(\tau) \| = \| \alpha\Delta(z) - \alpha(\tau) \|
< \epsilon. \]
We have just established that $\alpha(\tau)\in M(G)$ (inside $VN(G)$)
and so the result follows.
\end{proof}

Of course, we would like to prove that such $z$ are actually in
$C^*_\delta(G)$.  Let $M$ denote the norm closure of $M(G)$ in $VN(G)$.
Consider the closure of $\{ \pi(\mu)\xi_0 : \mu\in M \}$ in $H$.  We shall
shortly see that this Hilbert space is isomorphic to $\ell^2(G)$.  However,
we have been unable to decide if this is all of $H$ or not.  Furthermore, just
knowing that $\Delta(z) \in VN(G)\otimes VN(G)$ and that $z\in M$ does not
tell us that we can approximate $\Delta(z)$ by an element of $M\odot M$.
In the classical situation, when we compute the Bohr compactification of
$G$ (and not $\hat G$) then all our C$^*$-algebras are commutative, and so
they all have the approximation property, and so knowing, for example,
that $\omega\star z\in M$ for all $\omega$ does tell us that $z\in
M\otimes M$.  In our setting, working with operator spaces, it seems very
unlikely that $M$ will have the required operator space version of the
approximation property.  We consequently impose a slightly stronger condition;
see Theorem~\ref{thm:cap_partial_case} below.

Define $\theta:\ell^2(G)\rightarrow H$ by $\delta_s\mapsto \lambda(s)\xi_1$.
We note that
\[ \big( \lambda(s)\xi_1 \big| \lambda(t)\xi_1 \big)
= \lim_{V\rightarrow\mc U} \ip{\lambda(t^{-1}s)}{\omega_{\xi_V}}
= \delta_{t,s}, \]
because if $t^{-1}s$ is not the identity, then
$\ip{\lambda(t^{-1}s)}{\omega_{\xi_V}}$ will be zero for a sufficiently small
neighbourhood $V$.  It follows that $\theta$ is an isometry, and so $\theta
\theta^*$ is the orthogonal projection of $H$ into $\theta(\ell^2(G))$.

\begin{lemma}\label{lem:ell2_inv}
The action of $C_0(G)$ on $H$ leaves $\ell^2(G)$ invariant, and so
$\mu_{\theta(h),\eta} = \mu_{\theta(h),\theta\theta^*\eta} \in \ell^1(G)$
for all $h=(h_s)\in\ell^2(G)$ and $\eta\in H$.
\end{lemma}
\begin{proof}
For $f\in C_0(G)$, we find that
\begin{align*} \lim_{V\rightarrow\mc U} & \| f \lambda(s) \xi_V
- f(s) \lambda(s) \xi_V \|^2
= \lim_{V\rightarrow\mc U}
\int_G |f(t)\xi_V(s^{-1}t) - f(s)\xi_V(s^{-1}t)|^2 \ dt \\
&= \lim_{V\rightarrow\mc U} |V|^{-1}
\int_G |f(t) - f(s)|^2 \chi_V(s^{-1}t) \ dt.
\end{align*}
Now, $\chi_V(s^{-1}t)=0$ unless $t\in sV$, a small neighbourhood of $s$.
As $f$ is continuous, $|f(t) - f(s)|^2$ will be, on average, small on
the set $sV$, and hence the limit is zero.  It follows that $f\theta(\delta_s)
= f(s) \theta(\delta_s)$, and so
\begin{align*} \ip{f}{\mu_{\theta(\delta_s),\eta}}
&= f(s) \big( \theta(\delta_s) \big| \eta \big)
= f(s) \big( \theta(\delta_s) \big| \theta\theta^*\eta \big)
= \ip{f}{\mu_{\theta(\delta_s),\theta\theta^*\eta}},
\end{align*}
as required.
\end{proof}

\begin{lemma}
Let $\mu\in M(G)$, and treat $\mu$ as an operator in $VN(G)$.
Let $a\in\ell^1(G) \subseteq \ell^2(G)$ be the atomic part of $\mu$.
Then $\mu\xi_1 = \theta(a)$.
\end{lemma}
\begin{proof}
Let $\mu,\nu\in M(G)$ and note that $M(G)$ is a Banach $*$-algebra
(the $*$-operation is $\ip{\mu^*}{a} = \int \overline{a(s^{-1})} d\mu(s)$
for $a\in C_0(G)$) with the natural map $M(G)\rightarrow VN(G)$ a
$*$-homomorphism.  Then
\[ (\mu\xi_1|\nu\xi_1) = \varphi(\nu^*\mu)
= \lim_{V\rightarrow\mc U} \big( \nu^*\mu \xi_V \big| \xi_V \big)
= \lim_{V\rightarrow\mc U} \ip{\nu^*\mu}{(\omega_V\otimes\iota)(W^*)}. \]
If we let $a_V = (\omega_V\otimes\iota)(W^*) \in C_0(G)$, then it's easy
to see that $a_V(e)=1$ for all $V$ (with $e$ the unit of $G$) and that
for any open set $U$ containing $e$, eventually the support of $a_V$ is
contained in $U$.  As $\nu^*\mu$ is a regular measure, it follows that
\[ \lim_{V\rightarrow\mc U} \ip{\nu^*\mu}{(\omega_V\otimes\iota)(W^*)}
= (\nu^*\mu)(\{e\}), \]
the measure of the singleton $\{e\}$.  However, for any Borel set $E$
we have that (see \cite[Theorem~19.11]{hr})
\[ (\nu^*\mu)(E) = \int_G \nu^*(E s^{-1}) \ d\mu(s). \]
Now, for any $s\in G$, we have that $\nu^*(\{s^{-1}\}) = \overline{\nu(\{s\})}$
and there is hence a countable (or possibly finite) subset $F\subseteq G$ with
$\nu^*(\{s^{-1}\})=0$ if $s\not\in F$.  So the function $s\mapsto
\nu^*(\{s^{-1}\})$ is countably supported, and hence
\[ (\nu^*\mu)(\{e\}) = \int_G \nu^*(\{s^{-1}\}) \ d\mu(s)
= \sum_{s\in G} \overline{\nu(\{s\})} \mu(\{s\}). \]
If $a,b\in\ell^1(G)$ are the atomic parts of $\mu$ and $\nu$ respectively,
then it follows that $(\nu^*\mu)(\{e\}) = (b^*a)(\{e\})$ and hence
$(\mu\xi_1|\nu\xi_1) = (a\xi_1|b\xi_1)$.

It hence follows that
\[ \| \mu\xi_1 - a\xi_1 \|^2 = (\mu\xi_1|\mu\xi_1) - (\mu\xi_1|a\xi_1)
- (a\xi_1|\mu\xi_1) + (a\xi_1|a\xi_1) = 0. \]
The proof is then complete by observing that if $\iota:\ell^1(G)
\rightarrow\ell^2(G)$ is the formal identity, then $\theta\iota(a)=a\xi_1$.
\end{proof}

\begin{theorem}[Theorem~\ref{thm:cap_partial_case}]
Let $G$ be a [SIN] group, and let $x_0\in VN(G)$ be such that
$\Delta^2(x_0) \in VN(G)\otimes VN(G)\otimes VN(G)$.
Then $x_0\in C^*_\delta(G)$.
\end{theorem}
\begin{proof}
Define $\overline\alpha: VN(G)\vnten VN(G)\vnten VN(G) \rightarrow
VN(G)\vnten VN(G)$ by
\[ \ip{\overline\alpha(x)}{\omega\otimes\tau}
= \ip{\alpha( (\iota\otimes\iota\otimes\tau)(x) )}{\omega}. \]
We should justify why this makes sense.  Notice that for $f,g\in L^2(G)$,
\[ \big( \alpha(x)f \big| g \big) = \big( x [W^*(\xi_V \otimes f)] \big|
[W^*(\xi_V\otimes g)] \big), \]
where $[W^*(\xi_V\otimes f)]$ is an element of $(L^2(G\times G))_{\mc U}$.
It's now clear that $\alpha$ is completely bounded, and now standard
operator space techniques show that $\overline\alpha$ exists, and is completely
bounded with $\|\overline\alpha\|_{cb} \leq \|\alpha\|_{cb} = 1$.
For $x\in VN(G)$ and $\omega,\tau\in A(G)$,
\[ \ip{\overline{\alpha}\Delta^2(x)}{\omega\otimes\tau}
= \ip{\alpha\big( (\iota\otimes\iota\otimes\tau)\Delta^2(x) \big)}{\omega}
= \ip{\alpha\Delta\big( (\iota\otimes\tau)\Delta(x) \big)}{\omega}
= \ip{\Delta(x)}{\omega\otimes\tau}. \]
Hence $\overline{\alpha}\Delta^2(x) = \Delta(x)$.

For $\epsilon>0$, we can find $u = \sum_{i=1}^n a_i\otimes b_i\otimes c_i
\in VN(G)\odot VN(G)\odot VN(G)$ with $\|\Delta^2(x_0)-u\|<\epsilon$.
Then
\[ \Big\| x_0 - \sum_i \alpha\big( \alpha(a_i\otimes b_i) \otimes c_i \big)
\Big\| = \big\| x_0 - \alpha\overline\alpha(u) \big\| < \epsilon. \]
However, for each $i$ there is a measure $\mu_i$ with $\alpha(a_i\otimes b_i)
= \mu_i$.  Hence, by Proposition~\ref{prop:what_alpha_is},
\[ \sum_i \alpha\big( \alpha(a_i\otimes b_i) \otimes c_i \big)
= \sum_i \alpha( \mu_i \otimes c_i )
= \sum_i \mu_{\mu_i\xi_1, R(c_i)^*\xi_1}. \]
As $\mu_i\xi_1 \in \theta(\ell^2(G))$, by Lemma~\ref{lem:ell2_inv},
it follows that the sum defines a member of $\ell^1(G)$.  So $x_0$
can be norm approximated by elements of $\ell^1(G)$, that is, $x_0\in
C^*_\delta(G)$.
\end{proof}

We remark (without proof) that if $\Delta^2(x)\in VN(G)\otimes VN(G)\otimes
VN(G)$ then using that $\Delta_*:A(G)\proten A(G)\rightarrow A(G)$ is a
complete surjection, one can show that $A(G)\rightarrow VN(G);
\omega\mapsto \omega\star x$ is completely compact.  It follows, as in
\cite{runde}, that if $G$ is amenable or connected (so $VN(G)$ is injective)
then $x\in\CAP(A(G))$.  Conversely, if $\Delta(x)\in VN(G)\otimes VN(G)$
then $\Delta^2(x) \in VN(G\times G)\otimes VN(G) \cap VN(G)\otimes
VN(G\times G)$, but we don't see why $\Delta^2(x)$ need be in
$VN(G)\otimes VN(G)\otimes VN(G)$.  Obviously $x\in C^*_\delta(G)$ does
imply this, however.

\section{Further examples}\label{sec:egs}

In this section, we study various examples, and present some counter-examples
to conjectures in \cite{soltan}.

\subsection{Commutative case}

There is little to say here-- the categorical construction obviously
agrees with the usual \emph{strongly almost periodic} or \emph{Bohr
compactification}, \cite{berg, holm, holm1}.
Furthermore, in the commutative case, there is no distinction between
the reduced and universal case.

\subsection{Reduced quantum groups}\label{sec:cocomm_red}

In \cite[Question~2]{soltan}, So\l tan asked, in particular, if
$\mathbb{AP}(C_0(\G))$ is always a reduced compact quantum group, or in
our language, if $\mathbb{AP}(C_0(\G)) \rightarrow C(\G^\sap)$ is an
isomorphism.  In this section, we shall show that the answer is ``no'',
even if $\G$ is cocommutative.

As in the previous section, let $\G=\hat G$ for a locally compact group $G$.
Then the compactification of $\G$ is $\widehat{G_d}$, so $C(\G^\sap)
= C^*_r(G_d)$ and $C^u(\G^\sap) = C^*(G_d)$.  We follow the notation of,
in particular, \cite{bkls}, and write again $C^*_\delta(G)$ for the span of
the translation operators $\{ \lambda(s) : s\in G \}$ in $M(C_0(\G))=M(C^*_r(G))$.
Following \cite{bkls}, consider the following surjective Hopf $*$-homomorphisms:
\[ \xymatrix{ C^u(\G^\sap) = C^*(G_d) \ar@{->>}[r]^-{\Psi} &
\mathbb{AP}(C_0^u(\G)) \ar@{->>}[r]^-{\alpha} \ar@/_1pc/@{->>}[rr] &
\mathbb{AP}(C_0(\G)) = C^*_\delta(G) \ar@{->>}[r]^-{\Phi} &
C(\G^\sap) = C^*_r(G_d). } \]
Here we use the notation of \cite{bkls}, except that our map $\alpha$ is
denoted by $\Lambda$ there (which obviously clashes with other notation in
this paper).  Then \cite{bkls} proves the following (some of these
results also follow from work in \cite{bedos} and \cite{dr}):
\begin{itemize}
\item $\alpha$ is an isomorphism if and only if $G$ is amenable;
\item $\Phi\circ\alpha$ is an isomorphism if and only if $G_d$ is amenable;
\item $\Phi$ is an isomorphism if and only if $G$ contains an open subgroup
  $H$ with $H_d$ amenable;
\item if $G$ is a connected Lie group, then $\Psi$ is an isomorphism if
  and only if $G$ is solvable, if and only if $\Phi$ is an isomorphism.
  Recall that in this case, $G$ is solvable if and only if $G_d$ is amenable,
  \cite[Theorem~3.9]{pat}.
\end{itemize}

\begin{example}
In particular, we see that the compact quantum group $\mathbb{AP}(C_0(\G))$
is reduced if and only if $\Phi$ is an isomorphism.  Setting $\G$ to be the
dual of $SU(2)$ or $SO(3)$, we see that $\G$ is a discrete, cocommutative
quantum group, and $\Phi$ is not an isomorphism, as $G$ is connected,
but $G_d$ is not amenable (as it contains a free group,
\cite[Proposition~3.2]{pat}).
\end{example}

\begin{example}
Let $G$ be an amenable, connected Lie group with $G_d$ non-amenable
(again, $G=SU(2)$ or $SO(3)$ works), and set $\G=\hat G$.
Then $\mathbb{AP}(C_0^u(\G)) = \mathbb{AP}(C_0(\G))$, but these are not
equal to either $C_0^u(\G^\sap)$ nor to $C_0(\G^\sap)$.  Such compact quantum
groups, lying strictly between their universal and reduced versions, were
studied in \cite[Section~8]{ks}, so this example gives a whole family of
further ``exotic'' compact quantum group norms.  Furthermore, as again $\G$ is
a discrete quantum group, this answers in the negative a conjecture made
after \cite[Question~1]{soltan}, as $\mathbb{AP}(C_0^u(\G))$ is not
universal.
\end{example}

We finish this section by observing that we can prove something like
analogues for some of the above facts for general quantum groups.

There are many equivalent definitions of what it means for a general locally
compact quantum group $\G$ to be coamenable.  We shall choose the definition
that $\G$ is \emph{coamenable} if the counit is bounded on $C_0(\G)$,
see \cite[Section~3]{bt}.  That is, there is a state $\epsilon\in C_0(\G)^*$
with $(\iota\otimes\epsilon)\Delta = \iota$.  Then $\epsilon$ is unique, and
$(\epsilon\otimes\iota)\Delta=\iota$.  For a locally compact group $G$,
always $G$ is coamenable, while $\hat G$ is coamenable if and only if
$G$ is amenable.

\begin{proposition}\label{prop:some_amen_conds}
Let $\G$ be a locally compact quantum group and consider the maps
$\alpha:\mathbb{AP}(C_0^u(\G)) \rightarrow \mathbb{AP}(C_0(\G))$ and
$\Phi:\mathbb{AP}(C_0(\G)) \rightarrow C(\G^\sap)$.
Then:
\begin{enumerate}
\item\label{prop:some_amen_conds:one}
$\Phi\circ\alpha:\mathbb{AP}(C_0^u(\G)) \rightarrow C(\G^\sap)$ is
injective (and hence an isomorphism) if and only if $\G^\sap$ is
coamenable.
\item\label{prop:some_amen_conds:two}
Suppose that $\G$ is coamenable.  Then the natural map
$\Phi:\mathbb{AP}(C_0(\G)) \rightarrow C(\G^\sap)$ is injective (that is,
$\mathbb{AP}(C_0(\G))$ is reduced) if and only if $\G^\sap$ is coamenable.
\end{enumerate}
\end{proposition}
\begin{proof}
For (\ref{prop:some_amen_conds:one}) we note that $C_0^u(\G)$ always
admits a bounded counit $\epsilon_u$, see \cite[Section~4]{kus1}.
If $\Phi\circ\alpha$ is injective
then it's an isomorphism (as the dense Hopf $*$-algebras agree, see
discussion around Definition~\ref{defn:curlyap_of_g}). Thus the restriction
of $\epsilon_u$ to $\mathbb{AP}(C_0^u(\G))$ induces a bounded counit on
$C(\G^\sap)$ and so $\G^\sap$ is coamenable.  Conversely, if $\G^\sap$ is
coamenable then $C^u(\G^\sap) = C(\G^\sap)$ (see \cite[Theorem~2.2]{bmt})
and so as the canonical surjection $C^u(\G^\sap) \rightarrow C(\G^\sap)$
factors through $\Phi\circ\alpha$, it follows that $\Phi\circ\alpha$
is injective.

For (\ref{prop:some_amen_conds:two}), let $\epsilon\in C_0(\G)^*$ be the
counit, which exists as $\G$ is coamenable.
If $\Phi$ is injective, then it is an isomorphism, and the restriction
of $\epsilon$ to $\mathbb{AP}(C_0(G)) \cong C(\G^\sap)$ defines a bounded
counit on $C(\G^\sap)$, showing that $\G^\sap$ is coamenable.
Conversely, if $\G^\sap$ is coamenable then by \cite[Theorem~2.2]{bmt} the
Haar state on $\mathbb{AP}(C_0(G))$ is faithful and so $\Phi$ is injective.
\end{proof}

\subsection{Compact quantum groups}

The whole theory is designed to ensure that if $\G$ is a compact quantum
group, then it is its own compactification (compare
\cite[Section~4.3]{soltan}).  Of interest here are links with
Section~\ref{sec:ba_sec}.  The following is an improvement upon
Theorem~\ref{thm:pss_ap}, in that we make no assumption about the antipode.
The proof is very similar to \cite[Theorem~2.6(2)]{woro2}, where it is
shown that, when $\G$ is compact, if $a\in\mc{P}^\infty(\G) \cap C_0(\G)$, 
then $a\in\mc{AP}(C_0(\G))$.  We give the details, because the argument is
not long, and makes an interesting link with Theorem~\ref{thm:pss_ap}.
Our proof avoids use of $L^2(\G)$, and so \emph{maybe} holds promise of
extension to the non-compact case.

\begin{theorem}\label{thm:ss_okay_cmpt}
Let $\G$ be compact.  Then $\mc{P}^\infty(\G) = \mc{AP}(C_0(\G))$.
\end{theorem}
\begin{proof}
It suffices to show that $x\in \mc{P}^\infty(\G)$ is in $\mc{AP}(C_0(\G))$.
Let $\varphi$ be the (normal) Haar state on $L^\infty(\G)$.
By \cite[Section~1]{kvvn}, compare also \cite[Theorem~2.6(4)]{woro2},
we know that for $a,b\in L^\infty(\G)$,
\[ (\iota\otimes\varphi)(\Delta(a^*)(1\otimes b) \in D(S), \quad
S\big( (\iota\otimes\varphi)(\Delta(a^*)(1\otimes b)) \big)
= (\iota\otimes\varphi)((1\otimes a^*)\Delta(b)). \]
For $a\in L^\infty(\G)$, let $\omega_a\in L^1(\G)$ be the functional
$\ip{b}{\omega_a} = \varphi(a^*b)$ for $b\in L^\infty(\G)$.  As $\varphi$
is a KMS state, such functionals are dense in $L^1(\G)$.

That $x\in\mc{P}^\infty(\G)$ means that $\Delta(x)=\sum_{i=1}^n x_i\otimes y_i$
with $\{x_i\}$ and $\{y_i\}$ linearly independent sets.  Arguing as in the
proof of Theorem~\ref{thm:pss_ap}, we can find $(a_i)\subseteq L^\infty(\G)$
such that $\ip{y_j}{\omega_{a_i}} = \delta_{ij}$.  Thus for each $i$,
\[ S\big( (\iota\otimes\varphi)(\Delta(a_i^*)(1\otimes x)) \big)
= (\iota\otimes\varphi)((1\otimes a_i^*)\Delta(x))
= \sum_j x_j \ip{y_j}{\omega_{a_i}}
= x_i. \]
So $x_i \in D(S^{-1})=D(S)^*$.  Applying the same argument to $x^*$ shows
that $x_i \in D(S)$.  Applying the same argument to $\G^\op$ shows that
$y_i \in D(S)\cap D(S)^*$ for all $i$.

A close examination of the proof of Theorem~\ref{thm:pss_ap} shows that
knowing that $x_i,y_i\in D(S)\cap D(S)^*$ for all $i$ is enough for the
proof to work, and so $x\in\mc{AP}(C_0(\G))$, as required.
\end{proof}

We remark that Woronowicz asked in \cite{woro2} if it was essential to
focus on \emph{reduced} compact quantum groups for this result hold.
This was answered affirmatively in \cite[Remark~9.6]{ks}; in our language,
a compact quantum group $(A,\Delta)$ is constructed, and an element $a\in A$
is found, such that $\Delta(a)$ is a finite-rank tensor in $A\odot A$,
but $a \not\in \mc{AP}(A)$.

\subsection{Discrete quantum groups}

Discrete quantum groups were extensively studied in \cite{soltan}.
An important tool is the \emph{canonical Kac quotient} of a compact
quantum group, an idea attributed to Vaes.  Given a compact quantum
group $(A,\Delta)$, let $I$ be the closed ideal formed of all $a\in A$
such that $\tau(a^*a)=0$ for all traces $\tau$.  If $A$ admits no traces,
set $I=A$.  Let $A_\kac=A/I$ with $\pi:A\rightarrow A_\kac$ the quotient map.
Then
\[ \Delta_\kac(a+I) = (\pi\otimes\pi)\Delta(a) \qquad (a\in A), \]
is well-defined, and $(A_\kac,\Delta_\kac)$ becomes a compact quantum group.
It turns out that the Haar state is tracial, and so $(A,\Delta_\kac)$ is
a Kac algebra.  In particular, the dual of $(A,\Delta_\kac)$ is a unimodular
discrete quantum group, and hence has a bounded antipode.  All this is
explained in \cite[Appendix~A]{soltan}.

Let us now adapt the argument given in \cite[Section~4.3]{soltan} and
single out a key idea.  Let $\G$ be discrete and set
$(A,\Delta) = (C^u(\hat\G),\Delta^u_{\hat\G})$.  Then $(A_\kac,\Delta_\kac)$
is a compact Kac algebra, and so also its universal form, say $(C^u(\hat\H),
\Delta^u_{\hat\H})$ is Kac.  Let $\pi_u:C^u(\hat\G) \rightarrow
C^u(\hat\H)$ be the unique lift of $\pi:A\rightarrow A_\kac$, and let
$\hat\pi:C_0(\H) \rightarrow C_0(\G)$ be the dual
(recall that $\G$ and $\H$ are discrete, and so $C_0^u(\H)=C_0(\H)$
and so forth).

\begin{proposition}
For all $n$, the map $V\mapsto (\hat\pi\otimes\iota)(V)$ gives a surjection
from the set of $n$-dimensional unitary corepresentations
$V\in M(C_0(\H))\otimes\mathbb M_n$ to the set of
$n$-dimensional unitary corepresentations of $C_0(\G)$.
\end{proposition}
\begin{proof}
As $\hat\pi$ is a Hopf $*$-homomorphism, we need only prove surjectivity of
the map; namely, that if $U\in M(C_0(\G))\otimes\mathbb M_n$ is a unitary
corepresentation, then $U=(\hat\pi\otimes\iota)(V)$ for some suitable $V$.
Recall again the work of Kustermans in \cite{kus1}.  There is a unique
$*$-homomorphism $\phi:A\rightarrow\mathbb M_n$ with
$U = (\iota\otimes\phi)(\hat{\mc V}_{\G})$.  Furthermore, as $\G$ and $\H$
are discrete,
\[ \mc U_{\G} = \hat{\mc V}_{\G},\mc U_{\H} = \hat{\mc V}_{\H}
\implies (\iota\otimes\pi_u)(\hat{\mc V}_{\G})
= (\hat\pi\otimes\iota)(\hat{\mc V}_{\H}). \]
As $\mathbb M_n$ has a faithful trace, it is easy to see that there is
a unique $*$-homomorphism $\phi_0:A_\kac\rightarrow\mathbb M_n$ with
$\phi_0\circ\pi = \phi$.  Recall the reducing morphism $\Lambda^u_{A_\kac}:
C^u(\hat\H)\rightarrow A_\kac$, and set
$V = (\iota\otimes\phi_0\circ\Lambda^u_{A_\kac})
(\hat{\mc V}_\H)$, an $n$-dimensional unitary corepresentation of $C_0(\H)$.
Then
\[ (\hat\pi\otimes\iota)(V)
= (\iota\otimes\phi_0\circ\Lambda^u_{A_\kac}\circ\pi_u)(\hat{\mc V}_\G)
= (\iota\otimes\phi_0\circ\pi)(\hat{\mc V}_\G)
= U, \]
as required.
\end{proof}

Now, Corollary~\ref{corr:kacokay} shows that finite-dimensional
unitary corepresentations of $\H$ are automatically admissible (a result
not available to So\l tan) and so we get the following (which
So\l tan was able to prove by other means, see \cite[Theorem~4.5]{soltan}).
We will revisit this result below.

\begin{corollary}\label{corr:uni_implies_ad_disc}
Any finite-dimensional unitary corepresentation of a discrete quantum
group $\G$ is admissible.
\end{corollary}
\begin{proof}
Let $U$ be a finite-dimension unitary corepresentation of $C_0(\G)$.
Then $U = (\hat\pi\otimes\iota)(V)$ for an (automatically) admissible
unitary corepresentation $V$ of $C_0(\H)$.  Then, as in the proof of
Proposition~\ref{prop:same_red_uni}, $\overline{V}$ is similar to a
unitary corepresentation, say $X$.  A simple calculation then shows that
$\overline{U} = (\hat\pi\otimes\iota)(\overline{V})$ is similar to
$(\hat\pi\otimes\iota)(X)$ and hence admissible.
\end{proof}

\section{Open questions}

The most interesting open question seems to be:

\begin{conjecture}\label{conj:one}
For any $\G$, we have that $\mc{P}^\infty(\G) = \mc{AP}(\G)$.
\end{conjecture}

One obvious attack is suggested by Theorem~\ref{thm:pss_ap}: show that if
$x\in\mc{P}^\infty(\G)$ then automatically $x\in D(S)\cap D(S)^*$.
Woronowicz's argument, Theorem~\ref{thm:ss_okay_cmpt}, shows that this is
true for compact $\G$.

In the extreme case of one-dimensional corepresentations, the answer is
also affirmative.  To be precise, if $x\in L^\infty(\G)$ is a corepresentation,
meaning that $\Delta(x)=x\otimes x$, then automatically $x$ is unitary (and so
$x\in D(S)\cap D(S)^*$).  Two independent proofs are given in
\cite[Theorem~3.2]{dl} and \cite[Theorem~3.9]{kn}, but both proofs use,
for example, that also $x^*x$ is a
character, and there seems little hope of extending these sorts of arguments
to more general periodic elements.

A weaker conjecture is the following:

\begin{conjecture}\label{conj:two}
For any $\G$, the finite-dimensional unitary corepresentations of
$C_0(\G)$ are admissible.
\end{conjecture}

This is true for compact quantum groups from Woronowicz's work
(see \cite[Proposition~6.2]{woro2}, which is the key to showing that
the matrix elements of unitary corepresentations form a Hopf $*$-algebra,
in the compact case).  It is true for Kac algebras by
Corollary~\ref{corr:kacokay}, and is true for discrete quantum groups
by Corollary~\ref{corr:uni_implies_ad_disc}.  Recall that the final result
is proved by using the ``canonical Kac quotient'' of the dual: any
finite-dimensional $*$-representation of a compact quantum group factors
through a (compact) Kac algebra; but this technique is something special to the
compact case, and fails for discrete quantum groups, for example.
However, we wonder if some slightly different technique could be used to prove
the conjecture?  We note that in all computations of the quantum Bohr
compactification, one computes the finite-dimensional unitary corepresentations
of $C_0(\G)$ via computing the finite-dimensional $*$-representations of
$C_0^u(\hat\G)$, and then in each special case, it turns out that these
corepresentations are always admissible.

In the classical situation, consider the link between finite-dimensional
unitary representations $\pi$ of $G$ in $\mathbb M_n$, and group homomorphisms
from $G$ to compact groups.  Trivially, any such $\pi$ induces a group
homomorphism $G\rightarrow U(n)$; and the Peter-Weyl theory tells us that to
understand homomorphisms $G\rightarrow K$ for compact $K$,
it is enough (in some sense) to
know the finite-dimensional unitary representations of $G$.  This second point
was of course generalised by So{\l}tan in \cite{soltan}.  However, the first
point has links to Conjecture~\ref{conj:two}.  The following is easy to show,
as every finite-dimensional corepresentation of a compact quantum group is
admissible.

\begin{proposition}
Conjecture~\ref{conj:two} holds for $\G$ if and only if every
finite-dimensional unitary corepresentation $U$ of $\G$ factors through a
compact quantum group.
\end{proposition}

We remark that the quantum group analogues of the unitary groups are the
``universal'' quantum groups, in the sense of van Daele and Wang,
\cite{wang,vdw}.  If Conjecture~\ref{conj:two} is false, then there are
finite-dimensional unitary corepresentations of $\G$ which have nothing to
do with compact quantum groups: a very strange situation!

\vspace{5ex}

\noindent\emph{Author's Address:}
\parbox[t]{3in}{School of Mathematics\\
University of Leeds\\
Leeds\\
LS2 9JT}

\bigskip\noindent\emph{Email:} \texttt{matt.daws@cantab.net}

\end{document}